\documentclass[11pt]{article} 

\usepackage[utf8]{inputenc} 

\usepackage{graphicx} 
\usepackage{amsthm}
\usepackage[paperwidth=5.5in, paperheight=8.5in]{geometry}
\usepackage{geometry} 
\usepackage{graphicx} 

\usepackage{booktabs} 
\usepackage{array} 
\usepackage{paralist} 
\usepackage{verbatim} 
\usepackage{subfig} 
\usepackage{amsmath, amsthm, amssymb}
\usepackage{authblk}
\usepackage{mathtools}

\newtheorem{theorem}{Theorem}[section]
\newtheorem{lemma}{Lemma}[section]
\newtheorem{corollary}{Corollary}[section]
\newtheorem{definition}{Definition}[section]
\newtheorem{example}{Example\normalfont}[section]
\let\oldexample\example
\renewcommand{\example}{\oldexample\normalfont}

\usepackage{fancyhdr} 
\usepackage{sectsty}
\allsectionsfont{\sffamily\mdseries\upshape} 

\usepackage[nottoc,notlof,notlot]{tocbibind} 
\usepackage[titles,subfigure]{tocloft} 

\title{The Combinatorics of Tandem Duplication}
\author[1]{Penso-Dolfin L}
\author[1,2]{Greenman CD*}
\affil[1]{School of Computing Sciences, University of East Anglia, Norwich, UK, NR4 7TJ.}
\affil[2]{The Genome Analysis Center, Norwich Research Park, Norwich, NR4 7UH.}

\date{} 

\begin{document}
\maketitle

\allowdisplaybreaks

\begin{abstract}
Tandem duplication is an evolutionary process whereby a segment of DNA is replicated and proximally inserted. The different configurations that can arise from this process give rise to some interesting combinatorial questions. Firstly, we introduce an algebraic formalism to represent this process as a word producing automaton. The number of words arising from $n$ tandem duplications can then be recursively derived. Secondly, each single word accounts for multiple evolutions. With the aid of a bi-coloured 2d-tree, a Hasse diagram corresponding to a partially ordered set is constructed, from which we can count the number of evolutions corresponding to a given word. Thirdly, we implement some subtree prune and graft operations on this structure to show that the total number of possible evolutions arising from $n$ tandem duplications is $\prod\limits_{k=1}^n(4^k-(2k+1))$. The space of structures arising from tandem duplication thus grows at a super-exponential rate with leading order term $\mathcal{O}(4^{\frac{1}{2}n^2})$.
\end{abstract}

\let\thefootnote\relax\footnote{*Corresponding Author}


\section{Introduction}

Tandem Duplications (TD) occur when a region of DNA is duplicated and inserted adjacent to the original segment. This can be seen in Figure \ref{TD_Example}A where we start with five contiguous regions, labeled $ABCDE$. This is the original configuration and is termed the \emph{reference}. We then have a tandem duplication of $BCD$ to give sequence $ABCDBCDE$. We then have a duplication of region $DB$ to finally give $ABCDBDBCDE$. This is one process that has long been known to be implicated in the formation of gene clusters \cite{Ohno}, \cite{Nye} and more recently has been implicated in the formation of amplicons in cancer \cite{McBride}, \cite{Raphael1}, \cite{Raphael2}, \cite{Zhang}. In both cases Darwinian selection may be acting to increase the number of copies of a target gene.

Analyzing tandem duplication leads to some interesting combinatorics, some questions of which have been considered elsewhere \cite{Gascuel}, \cite{Yang}, \cite{Bertrand}. In particular, \cite{Gascuel} and \cite{Yang} count the number of TD trees consistent with sequences arising from a TD process. These methods generally assume, firstly, that breakpoints can be re-used, and secondly, that the full sequences (such as $ABCDBDBCDE$ above) are available. Neither of these two assumptions necessarily applies to all situations.

Firstly, a breakpoint in this context can mean the gap between two contiguous loci, such as a pair of genes in a gene cluster, which can cover a wide region and be implicated in more than one duplication event with reasonable probability, or it can mean the precise end points of the duplicated region, which are less likely to be implicated on more than one occasion. Modern sequencing (paired-end) data can resolve breakpoints to the basepair level and reveal tandem duplications to great precision, such as with cancer data \cite{McBride}. In such cases, when a tandem duplication occurs, two breakpoints are implicated in a presumably random process. The chance that precisely the same positions are subsequently implicated in another TD is likely to be small and assuming unique breakpoint use is reasonable in these circumstances. The questions considered in this work are restricted to the case of unique breakpoint use.

Secondly, a full TD sequence contains more information than we may be accustomed to. Typically, we know the reference sequence, the number of copies of each region, and pairwise connectivity. For example, in Figure \ref{TD_Example}B, instead of the sequence $ABCDBDBCDE$, we see that we have $[1,3,2,3,1]$ copies of the five originating regions. We refer to this as a \emph{CNV} (copy number vector). We also see that we have two types of \emph{somatic connection} that do not exist in the originating reference $ABCDE$; one connecting the end of $D$ to the beginning of $B$, the other connecting the end of $B$ to the beginning of $D$. We can represent this as a \emph{TD-Graph}, such as in Figure \ref{TD_Example}Ciii, where each node represents a reference region, the numbers at each node represent the number of copies, and the curved edges represent the somatic connections. This is representative of the information that is typically available from some sequencing experiments \cite{Greenman2} and represents the data at the end of the TD process. However, the genome at the start of the process will be represented by the simple TD-Graph in Figure \ref{TD_Example}Ci which will change every time we have a TD. We thus have a sequence of TD-Graphs, such as in Figures \ref{TD_Example}Ci,ii,iii, arising from a TD process. We refer to this as a \emph{TD-Evolution}.

\begin{figure}[t!]
\centering
\includegraphics[width=160mm]{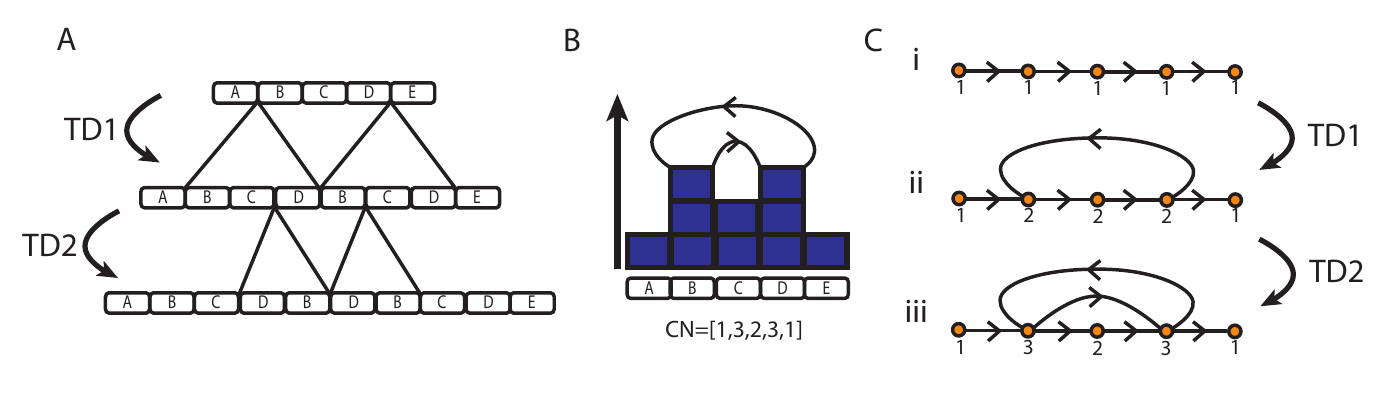}
\caption{A Tandem Duplication Process. A) Three TD sequences arising from two TDs on a reference of five regions; $ABCDE$. B) The \emph{copy number vector}, counting the number of copies of each reference region. C) The corresponding \emph{TD-Evolution} containing three \emph{TD-Graphs}; nodes represent reference regions, numbers at nodes count the number of copies of each region, and edges indicate connections between segments.}
\label{TD_Example}
\end{figure}

The problem we consider is concisely stated; count the number of different TD-Evolutions that arise from $n$ TDs. This is acheived as follows. Firstly, we consider how to best represent the process. We will see that by labeling each somatic connection with a number we can turn the process into an automaton acting upon words consisting of positive integers; any structure produced by $n$ TDs can then be represented by a word on symbols $1,2,...,n$. We then explore the size of this space of words. Each word will be seen to correspond to many different TD-Evolutions. Thus, secondly, we consider how to count the distinct TD-evolutions that all correspond to a single word. This involves the construction of a suitable partially ordered set (poset). Thirdly, we combine these two pieces of information and provide an explicit count of the number of possible evolutions for a given number of tandem duplications. Concluding remarks complete the paper.


\section{Representation}

We now introduce representations of the TD process, utilizing three different forms. First we have a visual \emph{zig-zag} representation, which is used to describe the final structure relative to the originating reference structure. Secondly we have an algebraic \emph{word evolution} representation, which enables the process to be viewed as an automaton on words composed of integers. Finally we have a \emph{$2$d-tree} representation. Now $2$d trees generalize the notion of trees. Trees can be characterized as connected graphs such that each node has a single parental node, apart from a single root node. We can define an $n$d tree to be a graph such that all nodes (except root nodes) have $n$ parental nodes. This kind of graph has been applied to data forms arising from search algorithms \cite{Bentley}, \cite{Turtle} and have seen other applications in genetics as recombination graphs \cite{Griffiths}, and pedigree graphs \cite{Kirkpatrick}, for example.

\begin{figure}[t!]
\centering
\includegraphics[width=160mm]{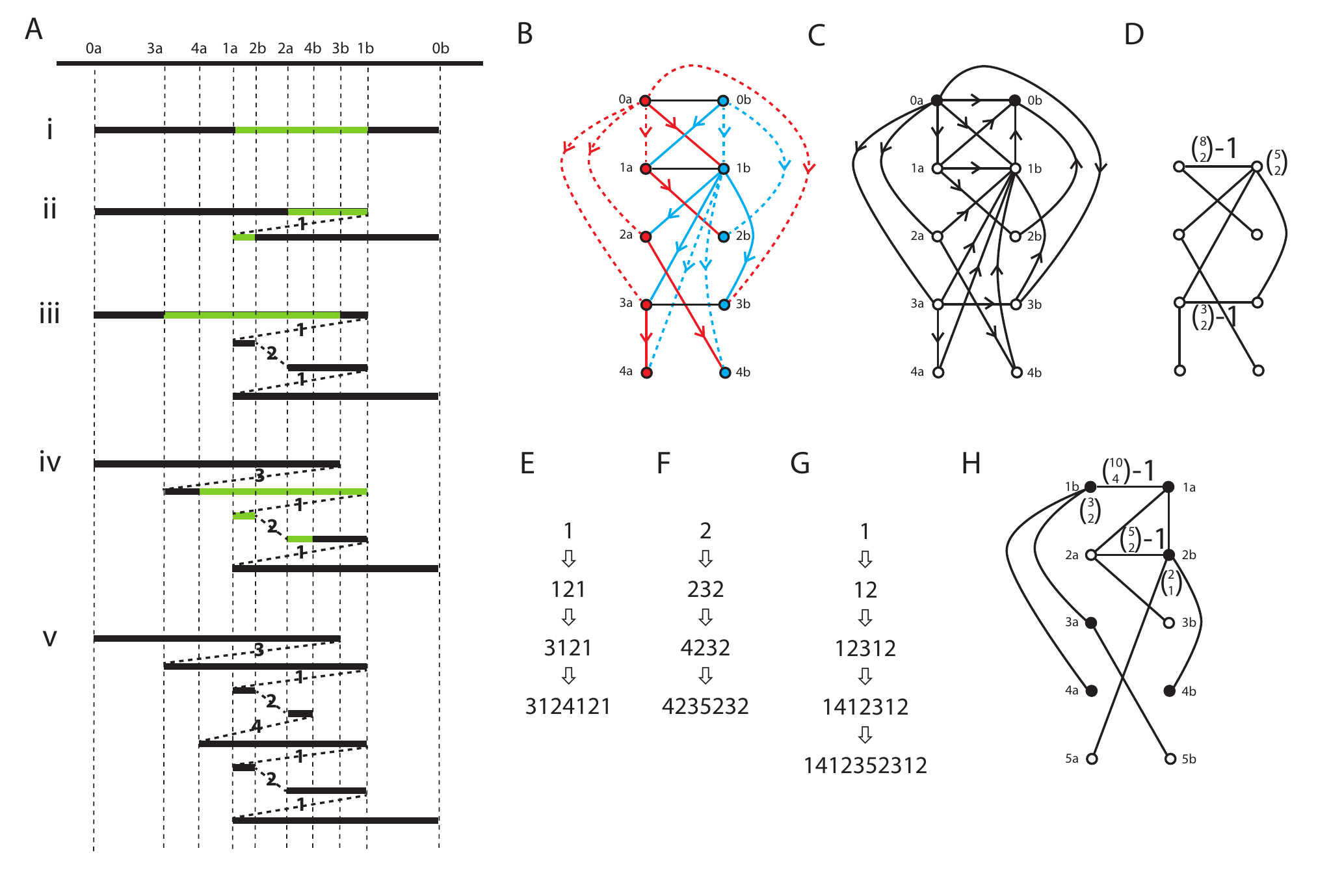}
\caption{Representation of the TD Process. In A) we have \emph{zig-zag} plots for a sequence of four TDs, resulting in five structures i)-v). The green regions indicates the region duplicated during each TD. Dashed lines indicate a connection between segments. Coordinates $n_a$ and $n_b$ indicate the end positions of the $n^{th}$ duplicated region. B) Corresponding 2d-tree. Nodes correspond to breakpoints and edges demarcate an ordering. Red and blue colours indicate lower and upper bound breakpoints. Dashed and plain edges indicate minor and major edges. C) Corresponding Hasse diagram. D) The major graph corresponding to evolution E. F) Increases each symbol of E by 1. G) An induced evolution from F. H) The major graph corresponding to induced evolution G. The black nodes indicate the corresponding $1$-nodeset.}
\label{TD_Process}
\end{figure}

We introduce the requisite structure with the example in Figure \ref{TD_Process}. We start with a single segment; an interval $[0_a,0_b]$, which is represented as a single horizontal line in Figure \ref{TD_Process}Ai. All coordinates described below are positioned relative to this interval, as demarcated by the positions at the top of Figure \ref{TD_Process}A. The node $0_a$ is assigned a type, $a$ (coloured red), indicating it is the left end of a segment. Node $0_b$ is assigned a type, $b$, (coloured blue) indicating a right end of a segment. These labels are associated with the top two nodes of the $2$d-tree in Figure \ref{TD_Process}B. These are bridged by an edge which will represent their ordering in the reference; $0_a < 0_b$.

Next we have the first TD event. This involves the duplication of a specific single region (coloured green in Figure \ref{TD_Process}Ai), and so implicates two positions; the left and right ends of the duplicated region, with reference coordinates $1_a$ and $1_b$, respectively. Now end $1_b$ is connected to $1_a$ in the duplication process, represented by the dashed line in the zig-zag diagram of Figure \ref{TD_Process}Aii. This is the first \emph{somatic connection}, labeled with numerical symbol $1$. Thus we have our first word of Figure \ref{TD_Process}E; $1$. Now the two positions $1_a$ and $1_b$ are both bound between the coordinates of $0_a$ and $0_b$. In the 2d-tree representation, we have two nodes representing coordinates $1_a$ and $1_b$. These nodes both have edges connected to two parental nodes $0_a$ and $0_b$. We have edges of type $a$ (red) from node $0_a$ to $1_a$ and $1_b$ representing the fact that $0_a$ is a lower bound of $1_a$ and $1_b$. Similarly we have edges of type $b$ (blue) from $0_b$ to $1_a$ and $1_b$, representing the fact that $0_b$ is an upper bound of $1_a$ and $1_b$. The black edge is a third class of edge, termed a \emph{fence}, and connects $1_a$ to $1_b$, representing the restriction $1_a < 1_b$. 

Our second TD then duplicates the green portion in Figure \ref{TD_Process}Aii, which includes the first somatic connection, forming two breakpoints $2_a$ and $2_b$. Position $2_a$ is on the upper segment $[0_a,1_b]$ of Figure \ref{TD_Process}Aii and so must lie between positions $0_a$ and $1_b$. These are its two parental nodes. The blue edge from node $1_b$ to $2_a$ indicates $1_b$ is an upper bound of $2_a$. The red edge from $0_a$ to $2_a$ indicates $0_a$ is a lower bound of $2_a$. The status of \emph{major} (solid) and \emph{minor} (dashed) is assigned to each pair of parental edges to a node, where major and minor refer to the parental nodes with higher and lower TD numbers, respectively. For example, $2_a$ has parents $1_b$ and $0_a$, the TD numbers satisfy $1>0$, so the edge from $1_b$ is the major, and that from $0_a$ is the minor. This distinction will later be important. This results in one new somatic connection, labeled $2$, and a duplication of the original connection $1$. Reading the somatic connections through the structure in Figure \ref{TD_Process}Aiii then produces the second word $121$ of Figure \ref{TD_Process}E.

We then procede through the TDs building up the 2d-tree. In general we have the following:
\\
\newline
\emph{\bf 2d-Tree Construction}
\\

Initialize with segment $[0_a,0_b]$. Let $n_a$ and $n_b$ represent the reference positions of the start and end of the $n^{th}$ duplicated region, where $n$ is the \emph{TD-number}. Node $n_a$ is designated type $a$ (coloured red), and $n_b$ is designated type $b$ (coloured blue). If $n_a$ (resp. $n_b$) lies on the segment $[u_a,v_b]$ we have a type $a$ (red) edge from $u_a$ to $n_a$ (resp. $n_b$), and a type $b$ (blue) edge from $v_b$ to $n_a$ (resp. $n_b$). If $u>v$, the edge from $u_a$ is major (solid), the edge from $v_b$ is minor (dashed). This is reversed if $u<v$. If $n_a$ and $n_b$ are formed on the same segment, no somatic connections are duplicated, and we must have $n_a<n_b$ which we represent with a fence edge (black) between nodes $n_a$ and $n_b$.
\\

Note that the choice of major and minor is ambiguous for the first TD. Both $1_a$ and $1_b$ are placed on the same interval $[0_a,0_b]$ so have parental nodes $0_a$ and $0_b$ that have equal TD-number $0$. It will prove consistent to define them as follows; $1_a$ has major (resp. minor) parental nodes $0_b$ (resp. $0_a$), $1_b$ has major (resp. minor) parental nodes $0_a$ (resp. $0_b$). Note that in all other cases either a type $a$ node $n_a$ is placed on $[u_a,v_b]$, where $n>\{u,v\}$ resulting in new interval $[n_a,v_b]$, with $n \ne v$, or a type $b$ node $n_b$ is placed on $[u_a,v_b]$, where $n>\{u,v\}$ resulting in new interval $[u_a,n_b]$, with $n \ne u$. Thus apart from the initial interval, the TD numbers of the endpoints of any interval are distinct and the major/minor is well defined.

We note that the relative order of positions $n_a$ and $n_b$ distinguishes two types of somatic connection. The somatic connection of the $n^{th}$ TD is \emph{reversed} when $n_a<n_b$ and otherwise it is \emph{forward}. Note, for example, that the first TD in Figure \ref{TD_Example}Cii is reversed ($1_a<1_b$) with a backward direction on the corresponding edge of the TD-Graph. This is the type of connection usually associated with TDs. However, the second TD is forward ($2_a>2_b$) with a forward direction on the corresponding (lower curved) edge of the TD-Graph in Figure \ref{TD_Example}Ciii . This is the type of connection usually associated with dna deletions, but can occur from multiple TDs.

Finally we comment that the presence of fences implies structures such as Figure \ref{TD_Process}B take a more general form than a 2d-tree. For convenience we use the phrase \emph{2d-tree} with that understanding in mind. 


\section{Word Representations}

We can describe the evolution in the example of Figure \ref{TD_Process}A in terms of words (Figure \ref{TD_Process}E); $\underline{1} \rightarrow 121 \rightarrow 3\underline{12}1 \rightarrow 3124121$ (duplicated subwords are underlined). Here, the second TD duplicates the first somatic connection, the third TD duplicates no somatic connections, and the fourth duplicates the sub-word of somatic connections $12$. This evolution of words is an example of an automaton \cite{Automata}:
\\
\newline
\emph{\bf TD Word Automaton}
\\
Initialize with word $W_1=1$. Then the word formed from the $n^{th}$ TD is obtained recursively as:

\begin{center}
$W_n=W_{n-1}(1:a-1)\cdot W_{n-1}(a:b)\cdot n\cdot W_{n-1}(a:b)\cdot W_{n-1}(b+1:N_{n-1})$ 
\end{center}

Here $N_k$ is the length of word $W_k$, $W_k(u:v)$ is the sub-word formed from the $u^{th}$ letter to the $v^{th}$ letter (inclusive), and $b \ge a-1$. If $b = a-1$ then $W_{n-1}(a:b)$ is empty. If $a=1$ then $W_{n-1}(1:a-1)$ is empty. If $b=N_{n-1}$ then $W_{n-1}(b+1:N_{n-1})$ is empty.
\\

We refer to the sequence $W_1 \rightarrow W_2 \rightarrow ... \rightarrow W_n$ as a \emph{word evolution} on $n$ TDs. We let $\mathcal{W}_n$ denote the set of all possible word evolutions on $n$ TDs.

We introduced fences for the situation where $n_a$ and $n_b$ form on the same segment. This means that the $n^{th}$ TD does not duplicate any somatic connections. In terms of the word automaton, these correspond to a step where no symbols are duplicated; no symbol is duplicated in the step $121 \rightarrow 3121$ for example.

We next consider how many words can arise from $n$ TDs, that is, the size of the space $\mathcal{W}_n$. For example, from the initial word $1$, a second TD can produce words $12$, $21$ or $121$, and $|\mathcal{W}_2|=3$; two words of length $2$ and one word of length $3$. In general we have the following result.

\begin{theorem}
If $w_{m,n}$ is the number of words of length $m$ arising from $n$ TDs, we have the following recursion,

$w_{m,n}=\sum_{k=\lfloor \frac{m-1}{2} \rfloor}^{m-1}(2k-m+2)w_{k,n-1}$

where we have initial values $w_{i,0}=\{\substack{1,i=0 \\ 0,i \ge 1}$
\label{LengthTheorem}
\end{theorem}
\begin{proof}
If we have a word with $k$ symbols then we can duplicate $r \in \{0,1,..,k\}$ of those symbols. Furthermore there are $k-r+1$ sets of $r$ consecutive symbols that we can choose to duplicate. Note that a TD duplication copies $r$ symbols and also introduces one new TD symbol, resulting in a word with $m=k+r+1$ symbols. Then $k=m-r-1$ for $r \in \{0,1,..,k\}$ and any word of length $m$ can derive from a word of length $k \in \{ \lfloor \frac{m-1}{2} \rfloor,...,m-1\}$. Lastly, we note that there are $k-(m-k-1)+1=2k-m+2$ ways to do this.
\end{proof}

\begin{figure}[t!]
\centering
\includegraphics[width=60mm]{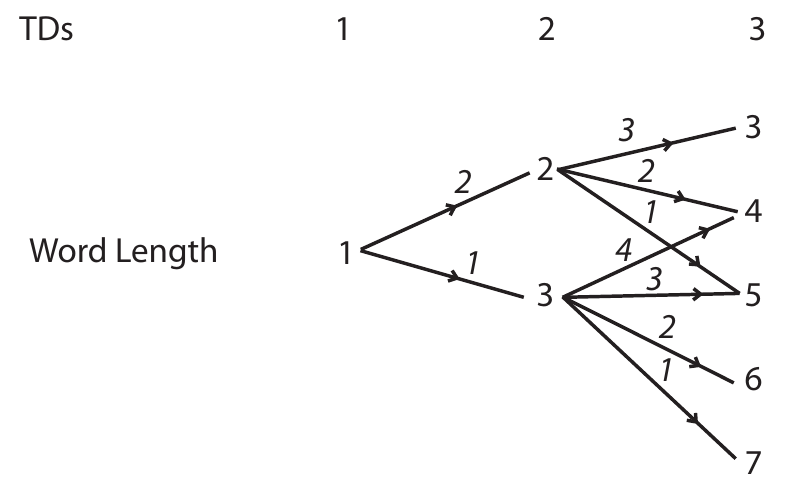}
\caption{Schematic of number of possible TD words. Numbers at nodes indicate the length of TD words. Numbers on edges indicate the number of choices.}
\label{TD_Words}
\end{figure}

\begin{example}
In Figure \ref{TD_Words} we see a graph representation of the possibilities, where values $w_{n,m}$ are equivalently obtained by taking products of the edge values along paths to the associated node, from the node labeled $0$, and summing. For example, the node labeled $5$ in the fourth column of nodes corresponds to $w_{3,5}$ and has two paths, one with product $1\cdot2\cdot1$, the other with $1\cdot1\cdot3$ and we find $w_{3,5}=2+3=5$, five words of length five; $12312$, $21321$, $13121$, $12321$ and $12131$.  
\end{example}

It is natural to attempt to find a general formula for the number of words arising from $n$ TDs by constructing a generating function from this recursion. However, this approach did not prove fruitful suggesting a closed form expression for the word count is not forthcoming.

The counts $|\mathcal{W}_n|=\sum_mw_{m,n}$ of words arising from $n$ TDs can be seen in Table \ref{TableCount}.

\begin{table}[htbp]
\begin{center}
\scalebox{0.9}{
  \begin{tabular}{| c || c | c | c | c | c | c | c| c| c| c |}
    \hline
    \text{TDs}					& 1 & 2 & 3 & 4   & 5       & 6 \\ \hline \hline
    \text{Words} 				& 1 & 3 & 22 & 377 & 15,315   & 1,539,281 \\ \hline
    \text{CNVs} 		& 1 & 7 & 225 & 27,839 & - & - \\ \hline
    \text{TD-Graphs} 		& 1 & 8 & 288 & 37,572 & - & - \\ \hline
    \text{TD-Evolutions} 				& 1 & 11 & 627 & 154,869 & 156,882,297 & 640,550,418,651   \\ \hline
	  \end{tabular}}
\end{center}
\caption{Counts of TDs, Words, CNVs, TD-Graphs and TD-Evolutions.}
\label{TableCount}
\end{table}


\section{Counting Evolutions With Posets}

Although we can count the number of words with relative ease, there maybe several different TD-Evolutions that correspond to a single word. We see from Figure \ref{TD_Evolution_Counts}A, for example, that there are three TD-Evolutions possible in the word evolution $[1 \rightarrow 12]$, where the second somatic connection follows one copy of the first. These cases can be phrased more familiarly in terms of TD sequences. Now two TDs result in four breakpoints $1_a$, $1_b$, $2_a$ and $2_b$ that divide the original interval $[0_a,0_b]$ into five regions $A$, $B$, $C$, $D$ and $E$. There are three choices that use two pairs of breakpoints uniquely such that the second somatic connection follows the first. If we underline the duplication and use `$|_i$' for the $i^{th}$ somatic connection; $A\underline{BCD}E \rightarrow ABCD|_1B\underline{C}DE \rightarrow ABCD|_1BC|_2CDE$, $A\underline{BC}DE \rightarrow ABC|_1B\underline{CD}E \rightarrow ABC|_1BCD|_2CDE$ and $A\underline{B}CDE \rightarrow AB|_1BC\underline{D}E \rightarrow AB|_1BCD|_2DE$. Counting copies of the five regions $A$, $B$, $C$, $D$ and $E$, we see that the three cases give rise to CNVs $[1,2,3,2,1]$, $[1,2,3,2,1]$ and $[1,2,1,2,1]$, respectively. Although the first two are equal, all three can be seen to have distinct TD-Graphs in Figure \ref{TD_Evolution_Counts}A.

These three cases have the following explanation in terms of breakpoint ordering. Once the first duplication has occurred, the two breakpoints $2_a$ and $2_b$ associated with the second TD need to be positioned. Now, the first TD requires $1_a<1_b$ resulting in segments $[0_a,1_b]$ and $[1_a,0_b]$ (see Figure \ref{TD_Process}Aii, for example). To obtain the word $12$ from word $1$ we find that we must not copy the first somatic connection, and both $2_a$ and $2_b$ must lie on the second segment $[1_a,0_b]$, so we have $1_a<2_a<2_b$. We then find that the three cases depend on whether $1b$ is less than, in between, or greater than $2_a$ and $2_b$. The three evolutions in Figure \ref{TD_Evolution_Counts}A i-iii then correspond to the three orders $1_a<2_a<2_b<1_b$, $1_a<2_a<1_b<2_b$ or $1_a<1_b<2_a<2_b$. 

These distinct orders represent possible breakpoint positions, subject to the restrictions $1_a<1_b$ from the first TD, and $1_a<2_a<2_b$ from the second TD. Articulating these restrictions more generally requires the construction of a suitable partially ordered set (\emph{poset}) \cite{Posets}. A poset is a set of elements with some order relationships between the elements. Posets are usually represented by a Hasse diagram. This is a directed graph where nodes represent the poset elements, and a directed edge between two nodes indicates an order relation between the two corresponding elements. Any single ordering of the elements that satisfies such a set of restrictions is known as a \emph{linear extension}. The Hasse diagram for any TD word evolution can be readily constructed from the corresponding 2d-tree as follows.

\begin{lemma} If the direction of the type $b$ (blue) edges are reversed in the 2d-tree, and fences are directed from $n_a$ to $n_b$ whenever they occur, a Hasse diagram with single source node $0_a$ and single sink node $0_b$ is obtained.
\label{HasseCor}
\end{lemma}

For example, in Figure \ref{TD_Process}B we see 2d-tree corresponding to the word evolution given in Figure \ref{TD_Process}E. In Figure \ref{TD_Process}C we see the same graph except the blue edge directions have been reversed, and the three fences are directed. Note that all fully extended, directed paths lead from $0_a$ to $0_b$. Any linear extension, such as $0_a<3_a<4_a<1_a<2_a<2_b<2_b<4_b<3_b<1_b<0_b$ at the top of Figure \ref{TD_Process}A, is satisifed by this Hasse diagram.

\begin{proof} (\emph{of Lemma \ref{HasseCor}})
When we add any node $x \in \{n_a,n_b\}$ to the 2d-tree, it has two parental nodes $u_a$ and $v_b$. By construction, the node $x$ represents a breakpoint that is placed on the segment $[u_a,v_b]$ with leftmost reference position $u_a$ and rightmost position $v_b$, thus we have the ordering $u_a < x < v_b$ in terms of reference position. Now the type $a$ edge directed from $u_a$ to $x$ represents the ordering $u_a < x$. We then select direction of the edges in the Hasse diagram to represent increasing reference position. Now $x < v_b$, so we require a directed edge from $x$ to $v_b$, which is obtained by reversing the direction of the type $b$ edge in the $2$d-tree from $v_b$ to $x$. Finally we note that if we have a fence, we are adding two position $n_a$ and $n_b$ to the same segment. We then have the additional ordering $n_a < n_b$ which is represented by the addition of a direction from $n_a$ to $n_b$. 
\end{proof}

\begin{figure}[t!]
\centering
\includegraphics[width=160mm]{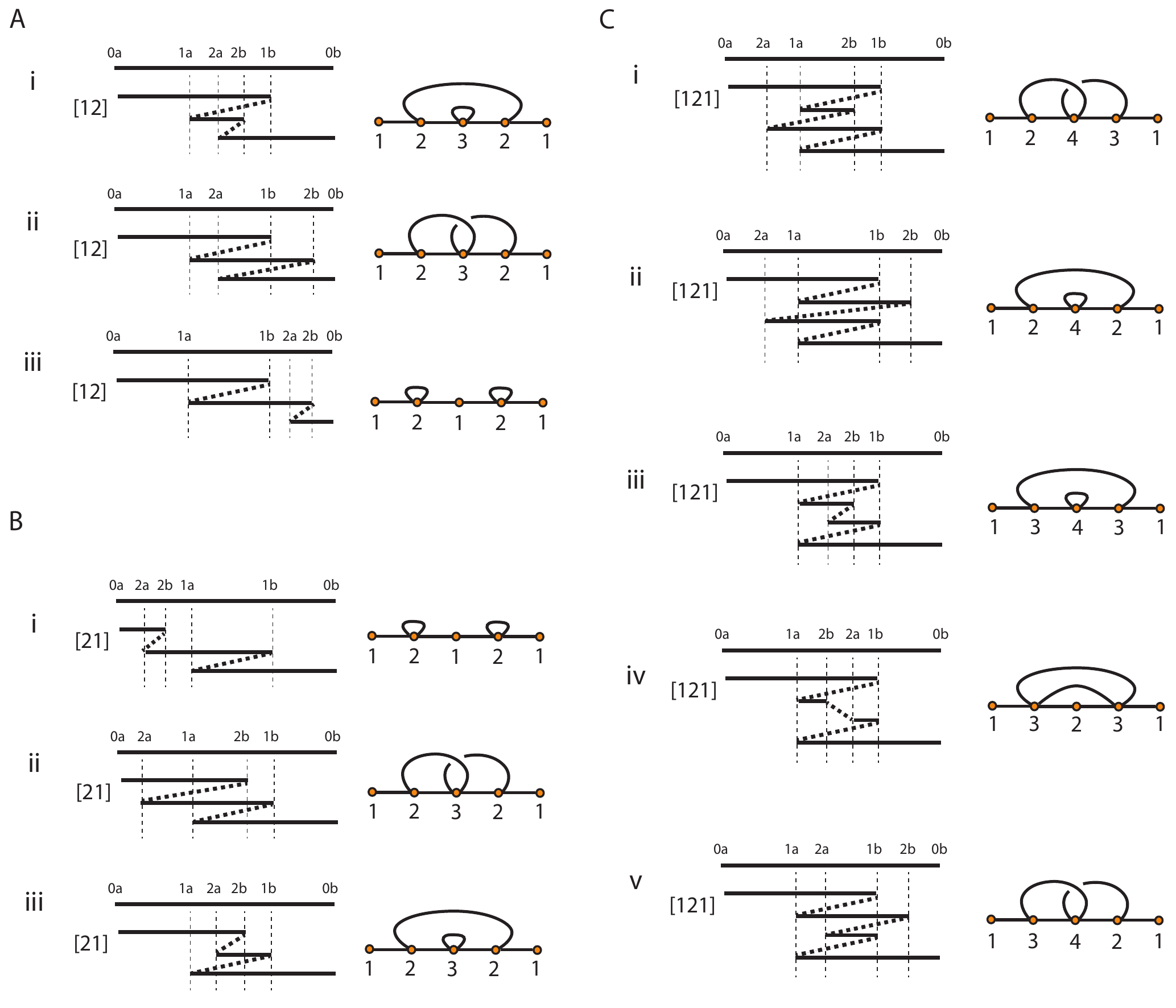}
\caption{Evolutions arising from two TDs. A) Three structures associated with word $12$. B) Three structures associated with word $21$. C) Five structures associated with word $121$. In each instance the left hand image is the zig-zig plot, and the right hand plot the TD-Graph.}
\label{TD_Evolution_Counts}
\end{figure}

Counting the number of different TD-Evolutions associated with a given word then reduces to counting the number of linear extensions associated with a poset. Although finding any single linear extension from a poset can be achieved in polynomial time \cite{Karzanov}, counting the number of linear extensions is known to be \#P-complete \cite{Brightwell} and in general is slow to implement \cite{Posets}. However, for the problem we have, we will show that restricting the Hasse diagram to major edges and fence edges (that is, removing the minor edges) contains all the ordering information. This simplified topology will enable us to obtain a closed form expression for the number of linear extensions. 

For any word evolution $E$, we will refer to the graph obtained from 2d-tree $T(E)$ by selecting just the major and fence edges as the \emph{major graph}, $T_{maj}(E)$.

The next result tells us that this simplified structure contains two trees if we also ignore the fences.

\begin{lemma}
The restriction of the Hasse diagram to the major edges results in two trees rooted to nodes $0_a$ and $0_b$. 
\end{lemma}
\begin{proof}
In the construction of the poset graph, every node $x \ne \{0_a,0_b\}$ has two parental nodes, labeled $u_a$ and $v_b$, arising from the segment $[u_a,v_b]$ that breakpoint $x$ is formed upon. These two nodes are connected to $x$ by a major and minor parental edge, where $max(u,v)$ and $min(u,v)$ are the major and minor TD numbers, respectively. Thus if we are restricted to the major edges, each node has one parental node, resulting in two trees attached to roots $0_a$ and $0_b$.
\end{proof}

The following result describes how major and minor status relates to the segments of the form $[u_a,v_b]$ involved in the TD process. It will be used to explain why removing the minor edges from the Hasse diagram does not lose any information.

\begin{lemma}
If $[u_a,v_b]$ is any segment arising in the evolution of a TD process then either:

A) Nodes $u_a$ and $v_b$ are connected by a single directed major edge from the node with TD number $min(u,v)$ to node with TD number $max(u,v)$. The positions satisfy the single linear extension $u_a<v_b$.

Or:

B) Nodes $u_a$ and $v_b$ are connected by a minor directed edge from the node with TD number $min(u,v)$ to that with TD number $max(u,v)$. Furthermore there exist nodes with TD numbers in the order $min(u,v)<n_1<n_2<...<n_I<max(u,v)$ that are connected in a chain of major edges in the same order such that: 

i) If $u>v$, all internal nodes are type $a$ (red) and the positions satisfy the single linear extension,

$(n_{1})_a<(n_{2})_a<...<(n_{I})_a<u_a<v_b$,

ii) If $u<v$, all internal nodes are type $b$ (blue) and the positions satisfy the single linear extension,

$u_a<v_b<(n_{I})_b<...<(n_{2})_b<(n_{1})_b$.
\label{CoreLemma}
\end{lemma}
\begin{proof}
We prove this by induction. Initially we start with a single segment $[0_a,0_b]$ and the first TD results in two segments $[0_a,1_b]$ and $[1_a,0_b]$ (such as in Figure \ref{TD_Process}Aii). Now node $1_b$ has major parental node $0_a$ and $1_a$ has major parental node $0_b$. Thus each of these segments has a single major edge connecting the corresponding nodes and so satisfy the conditions of the lemma.

For the induction we next assume that any segment $[u_a,v_b]$ satisfies the conditions of the lemma for all $u,v<m$. For each segment we thus have either a single major edge connecting nodes $u_a$ and $v_b$, or a minor edge connecting them along with a chain of major edges. We then introduce the $m^{th}$ TD duplicating a region with endpoints $m_a$ and $m_b$. We need to check all resulting segments satisfy the Lemma. We have four cases to check.

{\bf Case I}: The entire segment $[u_a,v_b]$ is duplicated or unmodified; then the poset graph is unchanged between nodes $u_a$ and $v_b$ and we have nothing to do.

{\bf Case II}: The breakpoint $m_a$ lies in $[u_a,v_b]$. We thus obtain a new segment $[m_a,v_b]$. A new node $m_a$ then has major and minor parents with TD number $max(u,v)$ and $min(u,v)$. We then have two possibilities depending on whether $u$ and $v$ are connected by a major or minor edge.

{\bf Case IIa}: If they are connected by a major edge then we see that if $u<v$ then we have a new major edge from $m_a \rightarrow v_b$, and segment $[m_a,v_b]$ satisfies criterion A of the Lemma. If $u>v$, then we have a minor edge $m_a \rightarrow v_b$ and a chain of two major edges $v_b \rightarrow u_a \rightarrow m_a$, which satisfy $u_a<m_a<v_b$, and segment $[m_a,v_b]$ matches criterion Bi of the Lemma.

{\bf Case IIb}: Now $u$ and $v$ are connected by a minor edge, along with a chain of major edges as described in the theorem. Then if $u<v$ we have a single major edge $v_b \rightarrow m_a$, and the conditions of the theorem are met. If $u>v$ we have a single minor edge $v_b \rightarrow m_a$ and major edge $u_a \rightarrow m_a$ with order $u_a<m_a<v_b$. If we combine this condition with the inductive hypothesis of the theorem; $(n_{1})_a<(n_{2})_a<...<(n_{I})_a<u_a<v_b$, we obtain $(n_{1})_a<(n_{2})_a<...<(n_{I})_a<u_a<m_a<v_b$, which again has the correct structure.

{\bf Case III}: If the breakpoint $m_b$ lies in $[u_a,v_b]$, a parallel set of reasoning to case II applies.

{\bf Case IV}: If both breakpoints $m_a$ and $m_b$ lie in $[u_a,v_b]$, we obtain segments $[u_a,m_b]$ and $[m_a,v_b]$. These are the same segments as cases II and III and the same arguments apply to both segments. 
\end{proof}

We now use this result to describe the inheritance nature of major and minor edges.

\begin{figure}[t!]
\centering
\includegraphics[width=160mm]{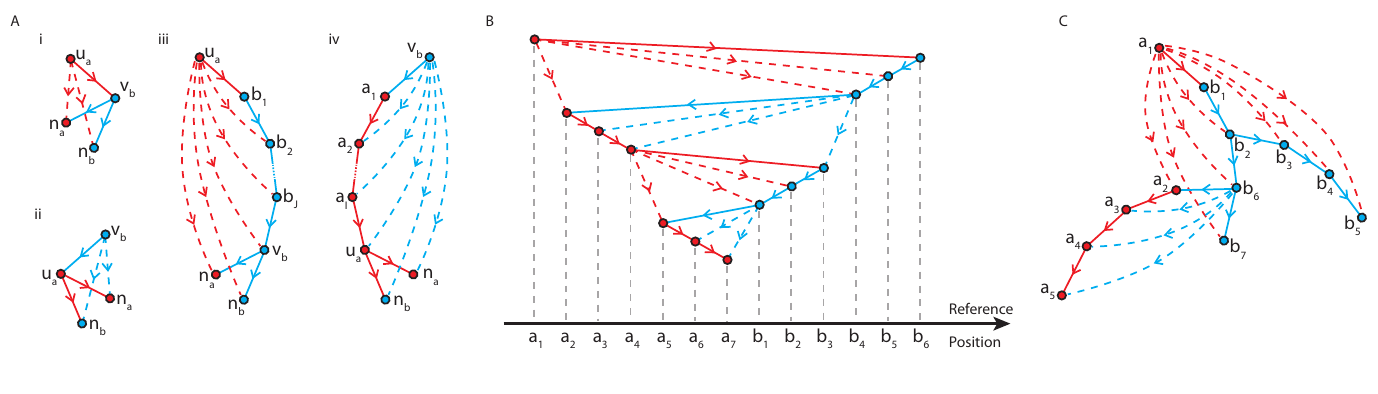}
\caption{Major and minor edge structure. A) The addition of new nodes preserves major-minor structure. B) The nesting structure of a branch of a major tree. C) The general major-minor structure.}
\label{TD_AB_Graphs}
\end{figure}

\begin{corollary}
If any node has a major parental node of type $a$ (resp. $b$), its minor parent is the most recent common ancestor (in the major graph) of opposite type $b$ (resp. $a$).
\label{AB_Structure}
\end{corollary}

\begin{example}
Consider the branch in Figure \ref{TD_AB_Graphs}C. Node $a_3$ has a major type $a$ parental node $a_2$. The most recent type $b$ ancestor of $a_3$ is node $b_6$, which is its minor parent. Node $b_5$ has a major type $b$ parental node $b_4$, we have to go back to node $a_1$ for its most recent type $a$ ancestor, its minor parental node.
\end{example}

\begin{proof} (\emph{of Corollary \ref{AB_Structure}})
Now by Lemma \ref{CoreLemma} any two nodes $u_a$ and $v_b$ bridging a segment $[u_a,v_b]$ are linked by a major or a minor edge. If a new node  $x \in \{n_a,n_b\}$ corresponding to a new breakpoint in this interval is formed, $u_a$ and $v_b$ are the major and minor parents, in some order. We have four cases to check:

{\bf Case I}: ($u<v$, major edge from $u_a$ to $v_b$). Then $x$ has minor parent $u_a$ and major parent $v_b$. The minor parent $u_a$ is then connected to $x$ by the chain of major edges $u_a \rightarrow v_b \rightarrow x$. Node $x$ has a major parent of type $b$ and the minor parent $u_a$ is the most recent ancestor of type $a$ in the major graph (see Figure \ref{TD_AB_Graphs}Ai).

{\bf Case II}: ($u>v$, major edge from $v_b$ to $u_a$). Analogous to Case I; swap $u$ and $v$, and swap $a$ and $b$ in argument (see Figure \ref{TD_AB_Graphs}Aii).

{\bf Case III}: ($u<v$, minor edge from $u_a$ to $v_b$). Then by Lemma \ref{CoreLemma} minor node $u_a$ is connected to major $v_b$ by a chain of major edges of the form $u_a \rightarrow (n_{1})_b \rightarrow (n_{2})_b \rightarrow...\rightarrow (n_{I})_b \rightarrow v_b$ for some internal nodes of type $b$. Now node $x$ has major parental node $v_b$ so there is also a major edge $v_b \rightarrow x$. Together we have the chain of major edges $u_a \rightarrow (n_{1})_b \rightarrow (n_{2})_b \rightarrow...\rightarrow (n_{I})_b \rightarrow v_b \rightarrow x$. We then find $x$ has a major of type $b$ and the minor $u_a$ is the most recent ancestor of type $a$ in the major graph (see Figure \ref{TD_AB_Graphs}Aiii).

{\bf Case IV}: ($u>v$, minor edge from $v_b$ to $u_a$). Analogous to Case III; swap $u$ and $v$, and swap $a$ and $b$ in argument (see Figure \ref{TD_AB_Graphs}Aiv).
\end{proof}

We can now explain the sense in which minor edges can be removed from the Hasse diagram. Specifically, we find that any set of nodes connected by a directed chain of major edges has a single ordering. More precisely:

\begin{corollary}
Consider any single directed chain of major edges connecting nodes $\{a_i,b_j:i=1,...,I,j=1,...,J\}$ where $a_i$ are nodes of type $a$ and $b_j$ are nodes of type $b$. Suppose furthermore that these nodes are in some order such that $a_i$ is an ancestor of $a_{i+1}$ for $i=1,2,...,I-1$, and $b_j$ is an ancestor of $b_{j+1}$ for $j=1,2,...,J-1$. These nodes have a single linear extension of the form:

$a_1<a_2<....<a_I<b_J<....<b_2<b_1$.
\label{SingleBranch}
\end{corollary}

Thus as we follow any single path down the major tree, the types $a$ and $b$ of the nodes can be intermixed. However, the TD numbers of the $a$ nodes increases down the path, as does the TD numbers of the $b$ nodes. Furthermore, the reference positions of the $a$ nodes increase and $b$ nodes decrease towards each other (see Figure \ref{TD_AB_Graphs}B for an example).

\begin{proof}
Now consider any sub-chain of nodes connected by major edges of the form $a_1 \rightarrow b_1 \rightarrow b_2 \rightarrow...\rightarrow b_n$. Then $b_{i+1}$ has major parent $b_i$ (of type $b$), so $b_{i+1}<b_i$. Also, $b_1$ has major parent $a_1$ (of type $a$) so $a_1<b_1$. We also know that $b_2,...,b_n$ all have $a_1$ as their minor parent by Corollary \ref{AB_Structure}, so $a_1<b_i$ for $i=2,3,...,I$. Together we then have the single order $a_1<b_n<...<b_2<b_1$. If the chain then continues as a chain of type $a$ nodes $b_n \rightarrow a'_1 \rightarrow a'_2 \rightarrow...\rightarrow a'_m$, we similarly find that $a'_1<a'_2<...<a'_m<b_n$. However, $a'_1$ has minor parent $a_1$ by Corollary \ref{AB_Structure} so $a_1<a'_1$. We then find that these two orders combine into the single order $a_1<a'_1<a'_2<...<a'_m<b_n<...<b_2<b_1$. Thus we find that as we move down a chain of nodes connected by major edges, the $a$ and $b$ nodes lie in one single nested structure where the $a$ nodes are increasing and the $b$ nodes are decreasing in reference position as we move down the major graph; a single linear extension.
\end{proof}

We now explain how to count the linear extensions using the major graph. In all that follows ${m \choose r}=\frac{m!}{r!(m-r)!}$ represent binomial coefficients, and ${m \choose m_1,...,m_I}=\frac{m!}{m_1!...m_I!}$ represent multinomial coefficients. There are two situations we need to deal with.

\begin{lemma}
i) Suppose $K$ branches descend from a single node $z$ in the major graph, such that the $k^{th}$ branch contains $m_k$ descendant nodes, and none of the $K$ daughter nodes of $z$ are connected by a fence. Then the number of linear extensions involving the associated $m+1$ breakpoints is ${m \choose m_1,...,m_K}\prod\limits_{k=1}^K\phi_k$, where $m=\sum\limits_{k=1}^K m_k$, and $\phi_k$ is the number of linear extensions associated with the $m_k$ nodes in branch $k$.

ii) Suppose two of the branches descending from a single node $z$ in the major graph contain $m_1$ and $m_2$ descendant nodes, respectively, and the two daughter nodes of $z$ in these branches are connected by a fence. Then the number of linear extensions involving the associated $m+1$ breakpoints is $({m \choose m_1}-1)\phi_1 \phi_2$, where $m=m_1+m_2$, and $\phi_1$ and $\phi_2$ are the number of linear extensions associated with the $m_1$ and $m_2$ nodes in the respective branches.
\label{CombinatoricsLemma}
\end{lemma}
\begin{proof}
i) We have $\phi_k$ linear extensions associated with branch $k$. If we select one linear extension from each branch, we have, by Corollary \ref{SingleBranch}, $K$ orderings of the form:

\begin{center}
$(x^{(k)}_{i_1})_a<(x^{(k)}_{i_2})_a<...<(x^{(k)}_{i_{m_k}})_b<(x^{(k)}_{i_{m_k+1}})_b$ 
\end{center}

Here $(x^{(k)}_{i_j})_{a/b}$ are the breakpoints represented by the nodes in branch $k$. Now node $z$ is the common ancestor of the $K$ branches and so arises from the earliest TD. Then by Corollary \ref{SingleBranch} either $z=(x^{(1)}_{i_1})_a=...=(x^{(K)}_{i_1})_a$ is the left most node and is of type $a$, or $z=(x^{(1)}_{i_{m_1+1}})_b=...=(x^{(K)}_{i_{m_K+1}})_b$ is the right most node and is of type $b$ (in Figure \ref{TD_AB_Graphs}B for example, the red node from the earliest TD is type $a$ and has the lowest position). Now node $z$ is fixed in position and common to all $K$ branches. Any pair of nodes from different branches are unrestricted relative to each other. Any pair of nodes within a branch $k$ have one relative order from the linear extension selected from the $\phi_k$ possibilities of that branch. We then need to count the number of ways of intercalating $m_1$ nodes from branch $1$, with $m_2$ nodes from branch $2$, through to $m_K$ nodes from the last branch. There are ${m \choose m_1,m_2,...,m_K}$ ways to do this.

ii) We now consider the case of a fence between two daughter nodes $n_a$ and $n_b$ of $z$, which results in the extra condition $n_a<n_b$. We have an ordering from each branch. By Corollary \ref{SingleBranch}, if $z$ is of type $a$ they will take the form:

\begin{center}
$(z)_a<n_a<(x^{(1)}_{i_{1}})_a<...<(x^{(1)}_{i_{m_1-1}})_b$

$(z)_a<(x^{(2)}_{i_{1}})_a<...<(x^{(2)}_{i_{m_2-1}})_b<n_b$ 
\end{center}

Here $(x^{(1)}_{i_j})_{a/b}$ and $(x^{(2)}_{i_j})_{a/b}$ are the breakpoints represented by the nodes descending from $n_a$ and $n_b$, respectively. Now there are ${m \choose m_1}$ ways to interlace these two orders. Furthermore, precisely one of these interlacements contradicts the extra condition $n_a<n_b$, and that is:

\begin{center}
$(z)_a<(x^{(2)}_{i_{1}})_a<...<(x^{(2)}_{i_{m_2-1}})_b<n_b<n_a<(x^{(1)}_{i_{1}})_a<...<(x^{(1)}_{i_{m_1-1}})_b$
\end{center}

We subtract this single order from the count ${m \choose m_1}$ to give the desired result.

The case where $z$ is of type $b$ is similar with the same conclusion.
\end{proof}

Finally we put this information together to count the number of linear extensions arising from the 2d-tree.

\begin{theorem}
Let the nodes $0_a$, $0_b$ and daughter edges be removed from the major graph. For each node $x$ remaining let $x_1,...,x_K$ denote the number of nodes that are present in each of $K$ descending branches. If any pair of daughter nodes are connected by a fence, they contribute a factor ${y_1+y_2 \choose {y_1}}-1$, where $y_1$ and $y_2$ count the number of nodes descending down each branch connected by the fence. These two branches are then treated as a single branch with $y_1+y_2$ daughter nodes. We then associate the number $m(x)={x \choose {x_1,...,x_r}}$ with node $x$. The number of distinct evolutions is then the product of these terms across nodes and fences.
\label{CountTheorem}
\end{theorem}
\begin{proof}
The TD process starts with segment $[0_a,0_b]$ which produces two segments $[0_a,1_b]$ and $[0_b,1_a]$ after the first TD. All future segments produced will always have at least one parental node with a TD number greater then $0$ so the only major edge from $0_a$ leads to $1_b$ and the only major edge from $0_b$ leads to $1_a$. Then $0_a$ and $0_b$ both have single branches descending. Now, applying Lemma \ref{CombinatoricsLemma} to any node with a single descending branch containing $n$ nodes results in a combinatorial term of the form $\frac{n!}{n!}=1$. The combinatorial factors from $0_a$ and $0_b$ can thus be ignored. For the remaining nodes we see from Lemma \ref{CombinatoricsLemma} that the orders $\phi_m$ associated with nodes in individual branches are multiplied into the combinatorial terms (such as ${m \choose m_1,...,m_K}$) associated with the parental node. We thus multiply the terms of the form   ${m \choose m_1,...,m_K}$ from nodes and ${m \choose m_1,...,m_K}-1$ from fences.
\end{proof}

\begin{example} Consider the word evolution $E=[1 \rightarrow 121 \rightarrow 3121 \rightarrow 3124121]$ with 2d-tree in Figure \ref{TD_Process}B. Once $0_a$ and $0_b$ are removed we have two fences corresponding to TD numbers $1$ and $3$. The restriction to major and fence edges then results in the graph in Figure \ref{TD_Process}D. The upper fence has two nodes attached to one side and six nodes to the other. This results in a count ${8 \choose 2}-1=27$. We note that node $1b$ has three branches descending; one fenceless branch with two nodes, and two branches bridged by a fence; one and two nodes down each branch. The latter two branches with the fence then have ${3 \choose 1}-1=2$ orders and are then treated as a single branch of three nodes. There are then ${5 \choose 2}=10$ ways of interlacing the five positions from the remaining branch with two nodes and amalgamated branch with three nodes. The total number of linear extensions, and so TD-evolutions, associated with word evolution $E$ is then $27\cdot 2\cdot 10=540$.
\end{example}

Note that in the proof we saw that a node with a single descending branch containing $n$ nodes results in a combinatorial factor $\frac{n!}{n!}=1$. This is true in general and explains why combinatorial terms from nodes with one descending branch were ignored in this example.

We thus now can count both the number of TD words, and the number of distinct evolutions for each word. We next consider how to combine this information and count the total number of evolutions for a given number of TDs. 


\section{The Size of TD Space}

We have seen that a TD process can be represented as an automaton on words. Furthermore, the number of TD-Evolutions represented by any single word evolution can be obtained from the corresponding major graph using the methods of the previous section. This naturally leads to the problem of determining the total number of TD-Evolutions. For example, in Figure \ref{TD_Evolution_Counts} we see all eleven evolutions that arise from two TDs; three evolutions corresponding to word $12$, three corresponding to $21$ and five corresponding to $121$. The aim of this section is to prove our main discovery:

\begin{theorem} The number $\mathcal{N}_n$ of distinct evolutions arising from $n$ TDs is given by:

\begin{center}
$\mathcal{N}_n=\prod\limits_{k=1}^n(4^k-(2k+1))$
\end{center}

\label{EvoCounter}
\end{theorem}

Thus $\mathcal{N}_2=(4^1-(2(1)+1))\cdot(4^2-(2(2)+1))=11$, in agreement with Figure \ref{TD_Evolution_Counts}, for example. The first few terms in this series can be seen in the bottom row of Table \ref{TableCount}. 


\subsection{A Motivating Example}

Before constructing a proof of Theorem \ref{EvoCounter}, we discuss a motivating example. Recall that $\mathcal{W}_n$ is the set of word evolutions on $n$ TDs. Consider the following examples.

\begin{center}
$E = [1 \rightarrow 121 \rightarrow 3121 \rightarrow 3124121]$

$E^{+} = [2 \rightarrow 232 \rightarrow 4232 \rightarrow 4235232]$

$E' = [1 \rightarrow 12 \rightarrow 12312 \rightarrow 1412312 \rightarrow 1412352312]$
\end{center}

The first two word evolutions both use four symbols; $E,E^{+} \in \mathcal{W}_4$. These only differ in the labeling of TDs; all we have done is increase each symbol in $E$ by $1$ to get $E^{+}$. In $E'$ we have a word evolution involving one more TD; $E' \in \mathcal{W}_5$. 

There are two things to note.

Firstly, if we delete the symbol $1$ in $E'$ we recover evolution $E^{+}$. That is, conversely, introducing a new first TD event to $E \in \mathcal{W}_4$ results in $E' \in \mathcal{W}_5$. This suggests we can generate TD-Evolutions in general by the repeated introduction of initial TDs. This leads to the following definition:

\begin{definition} If a new first TD is introduced to word evolution $E \in \mathcal{W}_n$, the resulting evolution $E' \in \mathcal{W}_{n+1}$ is called an \emph{induced} evolution.
\end{definition}

Secondly, the major graph of $E'$ is given in Figure \ref{TD_Process}H. Although we can form this directly from the word evolution $E'$ using the 2d-tree construction from the previous section, we note that Figure \ref{TD_Process}H is a subgraph of the 2d-tree from the original evolution $E$ (Figure \ref{TD_Process}B). This suggests we can get the major trees of induced evolutions from the 2d-trees of the originating evolutions.

This implies in general that there may be a connection between $\mathcal{W}_{n-1}$ and $\mathcal{W}_{n}$, both in terms of word evolutions, and in terms of major graphs. We need to explore both of these links in more detail.

Firstly we observe that for any word evolution there are a range of ways that a new first TD can be introduced. For example, take the trivial TD-Evolution $E=[1]$, and increase the symbols by $1$; $E^{+}=[2]$. We can introduce a new first TD in three ways; $E'=[1 \rightarrow 12]$, $E'=[1 \rightarrow 21]$ or $E'=[1 \rightarrow 121]$. Note that all three word evolutions reduce back to evolution $E^+=[2]$ if all copies of symbol $1$ are deleted. We will show something stronger in general; each single word evolution $E \in \mathcal{W}_{k-1}$ leads to a unique subset of induced evolutions $\mathcal{E} \subset \mathcal{W}_{k}$.

We will secondly show that all the major graphs for the word evolutions of $\mathcal{E}$ can be obtained from the 2d-tree for $E$. Now for any individual word evolution $E$, we can use the major graph $T_{maj}(E)$ to count the number of associated TD-Evolutions using Theorem \ref{CountTheorem}. We will extend this and show that the number of TD-Evolutions corresponding to $\mathcal{E}$ is equal to the number of TD-Evolutions corresponding to $E$ multiplied by a constant factor $4^n-(2n+1)$. Applying this observation recursively to the spaces $\mathcal{W}_{1},\mathcal{W}_{2},...,\mathcal{W}_{n}$ will then be seen to result in Theorem \ref{EvoCounter}.


\subsection{Induced Evolutions}

For induced evolutions to be a useful concept, we must establish that any word evolution $E' \in \mathcal{W}_{n+1}$ can be uniquely represented as an induced evolution from some word evolution $E \in \mathcal{W}_{n}$.

\begin{lemma} Let $D(E)$ be the process where we remove all copies of TD symbol $1$ from word evolution $E$ and reduce each symbol by $1$. This process has the following properties:

i) If $E \in \mathcal{W}_{n+1}$, then $D(E) \in \mathcal{W}_{n}$ is a valid word evolution.

ii) For any word evolution $E \in \mathcal{W}_{n}$, there exists a word evolution $E' \in \mathcal{W}_{n+1}$ such that $D(E')=E$.
\label{LemmaMap}
\end{lemma}
\begin{proof}
i) Any evolution $E$ starts with trivial word $1$. The next TD in $E$ results in word evolution $[1 \rightarrow 12]$, $[1 \rightarrow 21]$ or $[1 \rightarrow 121]$. For all three choices, removing the initial $1$ from the evolution leaves us the value $2$, which becomes $1$ when the symbols are reduced in value by $1$, thus we obtain the correct initial word for $D(E)$. Now the word evolution is constructed by the TD word automaton as a mapping of the form $AXB \rightarrow AX(n+1)XB$, for possibly empty subwords $A$, $X$ or $B$, for the $(n+1)^{th}$ TD. If we remove all copies of the symbol $1$ from the subwords $A$, $X$ and $B$, and reduce all symbols by $1$, to give $A'$, $X'$ and $B'$, respectively, we get a mapping of the form $A'X'B' \rightarrow A'X'nX'B'$ which is a valid step in the $n^{th}$ iteration of the TD word automaton, as required.

ii) For any evolution $E=[X_1 \rightarrow X_2 \rightarrow X_3 \rightarrow ... \rightarrow X_{n}]$ from $\mathcal{W}_{n}$ we simply construct $E'=[1 \rightarrow 1 X'_1 \rightarrow 1X'_2 \rightarrow 1X'_3 \rightarrow ... \rightarrow 1X'_{n}]$ where word $X'_i$ is obtained from $X_i$ by increasing each symbol by $1$. This is a valid word evolution in $\mathcal{W}_{n+1}$. Then applying $D$ to $E'$ recovers $E$, as required.
\end{proof}

This allows us to partition the space $\mathcal{W}_{n}$ as follows:

\begin{corollary} Let $\mathcal{E}(E)$ denote the set of induced evolutions from $E$. Then:

i) For any two evolutions $E_1,E_2 \in \mathcal{W}_{n}$, the two corresponding sets of induced evolutions do not overlap; $\mathcal{E}(E_1) \cap \mathcal{E}(E_2) = \phi$. 

ii) The set of induced evolutions satisfies the relation, $W_{n+1}=\bigcup\limits_{E \in W_{n}}\mathcal{E}(E)$.
\label{DisjointUnion}
\end{corollary}
\begin{proof}
i) We have shown from Lemma \ref{LemmaMap}i that deletion of symbol $1$ creates is a well defined mapping $D:\mathcal{W}_{n+1}\rightarrow \mathcal{W}_{n}$. Conversely, therefore, we therefore cannot have distinct word evolutions $E_1,E_2 \in \mathcal{W}_{n}$ that produce the same induced evolution $E'$ when a new first TD is introduced; $\mathcal{E}(E_1)$ and $\mathcal{E}(E_2)$ are thus distinct.

ii) We know from Lemma \ref{LemmaMap}ii that for any $E \in \mathcal{W}_{n}$, $\mathcal{E}(E) \in \mathcal{W}_{n+1}$. This implies that $\bigcup\limits_{E \in \mathcal{W}_{n}}\mathcal{E}(E) \subset \mathcal{W}_{n+1}$. Conversely, from Lemma \ref{LemmaMap}i we know that $\bigcup\limits_{E \in \mathcal{W}_{n}}\mathcal{E}(E) \supset \mathcal{W}_{n+1}$.
\end{proof}

Thus we can generate all of the word evolutions in $\mathcal{W}_{n+1}$ as a disjoint union of induced evolutions from $\mathcal{W}_n$. 

We wish to construct the major graphs $T_{maj}(E')$ of all the induced word evolutions $E' \in \mathcal{E}(E)$ from the 2d-tree $T(E)$ of the original evolution. To do this we need to relate the positions of new symbol $1$ in the new evolution $E'$ to the nodes of the 2d-tree $T(E)$. In all that follows $X$ represents unspecified subwords in a word evolution. We have the following definition.

\begin{definition}
Let $Z=\{1_a,1_b,2_a,2_b,3_a,3_b,...,n_a,n_b\}$ be the node labels for a 2d-tree $T(E^{+})$, where $E^{+}$ is the word evolution after the TD numbers have been increased by $1$ in some word evolution $E$. For any evolution $E'$ induced from $E$, a \emph{$1$-nodeset} $N \subseteq Z$ is defined as follows:

i) If the word $XmX$ in word evolution $E^{+}$ becomes word $X1mX$ in induced evolution $E'$, then $m_b \in N$.

ii) If the word $XmX$ in word evolution $E^{+}$ becomes $Xm1X$ in induced evolution $E'$, then $m_a \in N$.

iii) $1_a,1_b \in N$
\label{NodeSet}
\end{definition}

\begin{example} In Figure \ref{TD_Process}F,G we have evolutions: 

\begin{center}
$E^{+}=[{\bf 2} \rightarrow 2{\bf 3}2 \rightarrow {\bf 4}232 \rightarrow 423{\bf 5}232]$

$E'=[1 \rightarrow {\bf 12} \rightarrow 12{\bf 31}2 \rightarrow {\bf 141}2312 \rightarrow 14123{\bf 5}2312]$
\end{center}

Now $2$ in $E^{+}$ becomes $12$ in $E'$, so $2_b \in N$. Similarly, $X3X$ becomes $X31X$ so $3_a \in N$ (see bold symbols above). We see $X4X$ becomes $X141X$, the symbol $4$ picking up a $1$ either side in the induced evolution, so that $4_a,4_b \in N$. Finally we note that $5$ remains isolated from the symbol $1$ so $5_a,5_b \not\in N$. Thus $N=\{1_a,1_b,2_b,3_a,4_a,4_b\}$.
\end{example}

Now each \emph{$1$-nodeset} is a subset of the node labels for the 2d-tree. We find these sets have the following tree like structure:

\begin{lemma}
Let $T(E)$ be the 2d-tree for a word evolution $E$, and $N$ be the \emph{$1$-nodeset} corresponding to an induced evolution $E'$. Then if $x \in N$,

i) If $x$ is not a root node, its parents are in $N$.

ii) If $x$ is the parental node of a fence, at least one of the daughter nodes must be in $N$.

Conversely, any set of nodes $N$ from $T(E)$ satisying i) and ii) is a \emph{$1$-nodeset} for some induced evolution $E'$.

\label{NodeSetsNBetaTrees}
\end{lemma}

Thus the \emph{$1$-nodesets} have the tree like property that for any node belonging to the \emph{$1$-nodesets}, all its ancestors are also present. In particular, the root nodes belong to $N$. Consider the example above; $N=\{1_a,1_b,2_b,3_a,4_a,4_b\}$, these are the (solid) nodes in Figure \ref{TD_Process}H which satisfy these criterion. The two roots $1_a$ and $1_b$ are in $N$. There is a fence between $2_a$ and $2_b$, which have parental nodes $1_a,1_b$ that are members of node set $N$. At least one of $2_a,2_b$ must therefore be in $N$, and in this case $2_b$ is.

\begin{proof} (\emph{of Lemma \ref{NodeSetsNBetaTrees}})
Consider the $n^{th}$ TD in evolution $E$. We have two cases to consider. 

{\bf Case I}: The TD is not a fence. Then we have a node $n_a$ with parents $u_a$ and $v_b$, and node $n_b$ with parents $u'_a$ and $v'_b$. We also have step $XuvXu'v'X \rightarrow XuvXu'nvXu'v'X$ in the corresponding word evolution $E$, where the somatic connections from $v$ to $u'$ (inclusive) are duplicated. Note that subword $XuvX$ represents somatic connections across the region containing the breakpoint $n_a$, and $Xu'v'X$ similarly covers breakpoint $n_b$.

We consider changes to the parts $XuvX$ and $Xu'v'X$ of word $XuvXu'v'X$ when symbol $1$ is introduced in the induced evolution separately.

{\bf Case Ia}: If we have word $XuvX$ after symbol $1$ has been introduced into $E'$, then by Definition \ref{NodeSet}, $u_a,v_b \not\in N$. We then find we have evolution step $XuvX \rightarrow XuvXnvX$ in $E'$ and so symbol $n$ is not adjacent and left of symbol $1$ and we find that $n_a \not\in N$. That is, if the major parent of $n_a$ is not in $N$, $n_a$ cannot be in $N$.

If we have $Xu1vX$ after symbol $1$ has been introduced, then by Definition \ref{NodeSet}, $u_a,v_b \in N$. That is, the parents of $n_a$ are in $N$. Now we have two possibilities. Firstly, we can have evolution step $Xu1vX \rightarrow Xu1vXnvX$ in $E'$, where the somatic connection $1$ is not duplicated. In this case we find symbol $n$ is not adjacent to a $1$ so by Definition \ref{NodeSet}, $n_a \not\in N$. Secondly, we can have evolution step $Xu1vX \rightarrow Xu1vXn1vX$, where the somatic connection $1$ is duplicated. In this case we find symbol $n$ is adjacent and left of symbol $1$ so by Definition \ref{NodeSet}, $n_a \in N$. Thus if the parents of $n_a$ are in $N$, $n_a$ may or may not be in $N$ depending upon the choice of the induced evolution.

Note that the converse is also true and if the parents of $n_a$ are in $N$, we select the evolution step depending on whether $n_a$ is in $N$.

{\bf Case Ib}: The argument for node $n_b$, which depends upon $Xu'v'X$, is analogous, with the same conclusions.

{\bf Case II}: Consider the case that the $n^{th}$ TD results in a fence.

In that case we have a step $XuvX \rightarrow XunvX$ in $E$. Then if we have corresponding step $XuvX \rightarrow XunvX$ in induced evolution $E'$ we find that $u_a,v_b \not\in N$ by Definition \ref{NodeSet} and both $n_a,n_b \not\in N$.

Alternatively we may find that we have a step of the form $Xu1vX \rightarrow X$ in $E'$. Then parental nodes $u_a,v_b \in N$ and we have three possibilities to consider.

We may have $Xu1vX \rightarrow Xu1nvX$. Here the $1$ is not duplicated, but we find $n$ is to the right of a $1$ and so $n_b \in N$.

We may have $Xu1vX \rightarrow Xun1vX$. Here the $1$ is not duplicated, but we find $n$ is to the left of a $1$ and so $n_a \in N$.

Lastly, we may have $Xu1vX \rightarrow Xu1n1vX$. Here the $1$ is duplicated and both $n_a,n_b \in N$.

Thus when the parent nodes of a fence are in $N$, at least one of the daughter nodes $n_a$ or $n_b$ must be.

Conversely, when the parent node and one or more of $n_a$ or $n_b$ are in $N$, we select the corresponding evolutionary step.
\end{proof} 

We can now show how to construct the major graph $T_{maj}(E')$ from the parental 2d-tree $T(E)$. 

\begin{lemma}
For any evolution $E'$ induced from $E$, let $N$ be the corresponding $1$-nodeset obtained from Lemma \ref{NodeSetsNBetaTrees}. Let $T(E)$ be the 2d-tree corresponding to $E$. The major graph $T_{maj}(E')$ is constructed as follows.

i) Select all nodes from $T(E)$ and increase each TD number in the node labels by $1$. 

ii) If any type $a$ (resp. $b$) node (that is not a root) is a member of $N$, select the parental edge of the same type, $a$ (resp. $b$).

iii) If any type $a$ (resp. type $b$) node (that is not a root) is immediately adjacent (but not in) $N$, select the parental edge from the opposite type, $b$ (resp. $a$).

iv) If any node $n_a$ (resp. $n_b$) is neither a member of, or immediately adjacent to, $N$, select the major edge of $n_a$ from $T(E)$.

v) If $n_a$ and $n_b$ are connected by a fence in $T(E)$ (for TD number $n \ge 2$), select the fence if and only if $n_a \not\in N$ or $n_b \not\in N$. 

vi) Place a fence between $1a$ and $1b$ and swap these two node labels.
\label{TreeOperations}
\end{lemma}

\begin{example} Consider again the original 2d-tree $T(E)$ in Figure \ref{TD_Process}B, where $E$ is the evolution in Figure \ref{TD_Process}E. We wish to construct and the major graph $T_{maj}(E')$ (of induced evolution $E'$) given in Figure \ref{TD_Process}H by applying the Lemma. We found the $1$-nodeset corresponding to evolution $E'$ previously as $N=\{1_b,1_a,2_b,3_a,4_a,4_b\}$, the black nodes in Figure \ref{TD_Process}H. Then to construct $T_{maj}(E')$ we take the nodes of $T(E)$ and first increase the TD numbers by $1$, swap the two labels with TD number $1$, and place a fence between them. Node $2_a \not\in N$ is adjacent to $N$ so we select the edge from the node of opposite $b$ type by Lemma \ref{TreeOperations}iii. This was $0_b$, which is now mapped to $1_a$, so we select edge $1_a \rightarrow 2_a$. Node $2_b \in N$, so we select the parental edge of same node type $b$. This was also $0_b$, so we select the edge from mapped node $1_a \rightarrow 2_b$. By Lemma \ref{TreeOperations}v, we furthermore select the fence between nodes $2_a$ and $2_b$. Now $3_a \in N$ thus we select the edge from type $a$ parent, the node $1_b$ (mapped from $0_a$). Node $3_b$ is not in or adjacent to $N$ so we select the major edge from $T(E)$; $2_a \rightarrow 3_b$. Now $4_a,4_b \in N$ so we select the edges from parental nodes of same type; $1_b \rightarrow 4_a$ and $2_b \rightarrow 4_b$ parental node $2_b$. Nodes $5_a,5_b$ are adjacent to $N$ so we select its parent edges of opposite type; $2_b \rightarrow 5_a$ and $3_a \rightarrow 5_b$.
\end{example}

Observe that the differences between $T(E)$ and $T_{maj}(E')$ are a form of subtree prune and graft operations \cite{SempleSteel}; when the major edge is swapped for the minor edge we are pruning from the major parental node and grafting to the minor parental node.

\begin{proof}(\emph{of Lemma \ref{TreeOperations}})

i) All the breakpoints from evolution $E$ remain in evolution $E'$ so we inherit the representative breakpoints. The introduction of a new first TD increases each TD number by $1$.

ii) Consider the case that we have a type $a$ node $n_a \in N$. Then we have evolution step $XuvX \rightarrow XuvXnvX$ in $E^{+}$, and $n_a$ has major and minor parents $u_a$ and $v_b$ in $T(E)$ (in some order, depending upon whether $u>v$). This evolution step becomes $Xu1vX \rightarrow Xu1vXn1vX$ in $E'$. Then the somatic connection $1$ is duplicated and breakpoint $n_a$ occurs between somatic connections $u$ and $1$. The major and minor parents are then $u_a, 1_b$ in $T(E')$. Now $u>1$ so the major parent of $n_a$ in $T(E')$ is the node $u_a$. Thus if we have a type $a$ node $n_a \in N$, we select the parental edge of the same type; $u_a \rightarrow n_a$, irrespective of whether it was the major or minor in $T(E)$. The argument for $n_b$ is analogous.

iii) Consider the case that $n_a$ is adjacent to a node in $N$, that is, its major and minor parents are in $N$. Then we have evolution step $XuvX \rightarrow XuvXnvX$ in $E$ that becomes $Xu1vX \rightarrow Xu1vXnvX$ in $E'$. This time, in the induced evolution, the major and minor parents of $n_a$ are $1_a, v_b$. Now $v>1$ so the major parent of $n_a$ in $E'$ is node $v_b$. Thus if we have a type $a$ node $n_a \in N$, we select the parental edge of the opposite type; $v_b \rightarrow n_a$, irrespective of whether it was the major or minor in the original evolution $E$. The argument for $n_b$ is analogous.

iv) Consider the case that $n_a$ is neither adjacent to a node in, or a member of $N$. Then we have evolution step $XuvX \rightarrow XuvXnvX$ in $E$ that becomes $XuvX \rightarrow XuvXnvX$ in $E'$. Now the major/minor status of $n_a$ does not change from the original. The argument for $n_b$ is analogous.

v) If $n_a$ and $n_b$ are connected by a fence we have a step of the form $XuvX \rightarrow XunvX$ in $E$. The corresponding step in the induced evolution $E'$ takes one of four forms. Firstly, if $Xu1vX \rightarrow Xu1n1vX$ in $E'$, then $n$ is adjacent to $1$ on both sides, so $n_a,n_b \in N$. Note that $n$ has duplicated symbol $1$, so we do not have a fence in $T_{maj}(E')$. Secondly, if we have $Xu1vX \rightarrow Xu1nvX$ in $E'$, then $n_a \not\in N$ and $n_b \in N$. Note that $n$ has not duplicated the symbol $1$ and we still have a fence. Thirdly, the evolution $Xu1vX \rightarrow Xun1vX$ in $E'$ similarly preserves the fence, with $n_a \in N$ and $n_b \not\in N$. Finally, if we have $XuvX \rightarrow XunvX$ in $E'$, the fence is preserved and $n_a,n_b \not\in N$. Thus $T_{maj}(E')$ contains the fence if and only if at least one of $n_a \not\in N$ or $n_b \not\in N$ is true.

vi) Firstly note that the initial TD in any evolution must occur on the single reference segment, and so must be fence because there are no prior TDs to duplicate, thus we place a fence between nodes $1_a$ and $1_b$. 

Consider the $n^{th}$ TD for some $n \ge 2$. Note that the only way that node $n_a$ or $n_b$ can have a parental node $0_a$ or $0_b$ is to have a step of the form $X \rightarrow nX$ or $X \rightarrow Xn$ in word evolution $E$. Consider first the step $X \rightarrow nX$. Note that $n$ must represent a fence because there are no symbols to the left of $n$ which could have been duplicated. Then $n_a$ and $n_b$ have minor parents $0_a$ and some major parent $v_b$. The induced evolution can then be in one of three forms.

Firstly, we can have corresponding step $1X \rightarrow n1X$ in $E'$. In this case, from Definition \ref{NodeSet} we find $n_a$ is in $N$ and so is connected to its type $a$ parent $0_a$ by ii) above. Now because $n$ is a fence, $n_b$ has the same parents as $n_a$, so is adjacent to $N$ in $T(E)$. Then using iii) above we find $n_b$ is connected to its type $a$ parent, also $0_a$. However, constructing the major tree directly from the 2d-tree corresponding to word evolution $E'$, we find that $n$ is a fence with major parent $1_b$. Thus to get an equivalent form from the original 2d-tree, we map $0_a$ to $1_b$.

The case for $1X \rightarrow 1nX$ is similar, resulting in a map from $0_b$ to $1_a$.

For the third choice, the step becomes $1X \rightarrow 1n1X$ in $E'$. We then find that  $n_a,n_b \in N$ by Definition \ref{NodeSet} as $n$ is adjacent to $1$ on both sides. Thus in the major graph for $E'$, $n_a$ is connected to its type $a$ parent $0_a$ and $n_b$ is connected to its $b$ parent $u_b$. However, direct from $E'$ we see that $n_a$ has major parent $1_b$ and $n_b$ has major parent $u_b$, so again we map $0_a$ to $1_b$ for a consistent correspondence.

The argument using step $X \rightarrow Xn$ and node $n_a$ from $E$ is entirely similar with parallel conclusions.
\end{proof}

In summary, we now know that for any word evolution $E \in \mathcal{W}_{n-1}$ there is a unique subset of induced evolutions $\mathcal{E}(E) \in \mathcal{W}_{n}$, each member $E'$ of which corresponds to a $1$-nodeset from the 2d-tree of $E$. We can now use this to produce the major graph $T(E')$ for the induced evolution using Lemma \ref{TreeOperations}. We can then calculate the number of TD-Evolutions associated with each $E'$ from Theorem \ref{CountTheorem}. We thus need to sum the TD-Evolution counts across the set of $1$-nodesets corresponding to $\mathcal{E}(E)$. Whilst this is possible, leading to $4^n-(2n+1)$ induced TD-Evolutions for each word evolution $E$, the proof relies on a more general space of graphs than we have considered so far, which we now introduce.


\subsection{$\beta$-trees}

Firstly we generalize the notion of the 2d-tree obtained from TDs.

\begin{definition}
A \emph{$\beta$-tree} $T$ is any directed graph such that:

i) All nodes and edges are classified as either type $a$ or type $b$

ii) There is a root node ($A$) of type $a$ and a root node ($B$) of type $b$, and all directed edges point away from the roots.

iii) All other nodes have a type $a$ parental node and a type $b$ parental node. The two edges from the parental nodes are also of type $a$ and $b$, respectively. Either the two parents are the two roots, or one parent is a descendant of the other. The edge from the more recent ancestor is the major, the other is the minor.

iv) A type $a$ node and type $b$ node may be linked by a fence if they have the same parental nodes, or are the two roots.
\end{definition}

Thus the 2d-trees defined from TDs are $\beta$-trees, for example. Note that $\beta$-trees are more general; take Figure \ref{TD_Tree_Operations}A,B, for example, they do not have an even number of nodes and cannot arise from a TD process, but satisfy the requirements of a $\beta$-tree. 

Similar to the 2d-tree construction, the major graph $T_{maj}(T)$ of a $\beta$-tree $T$ is the graph obtained when the minor edges are removed.

\begin{figure}[t!]
\centering
\includegraphics[width=140mm]{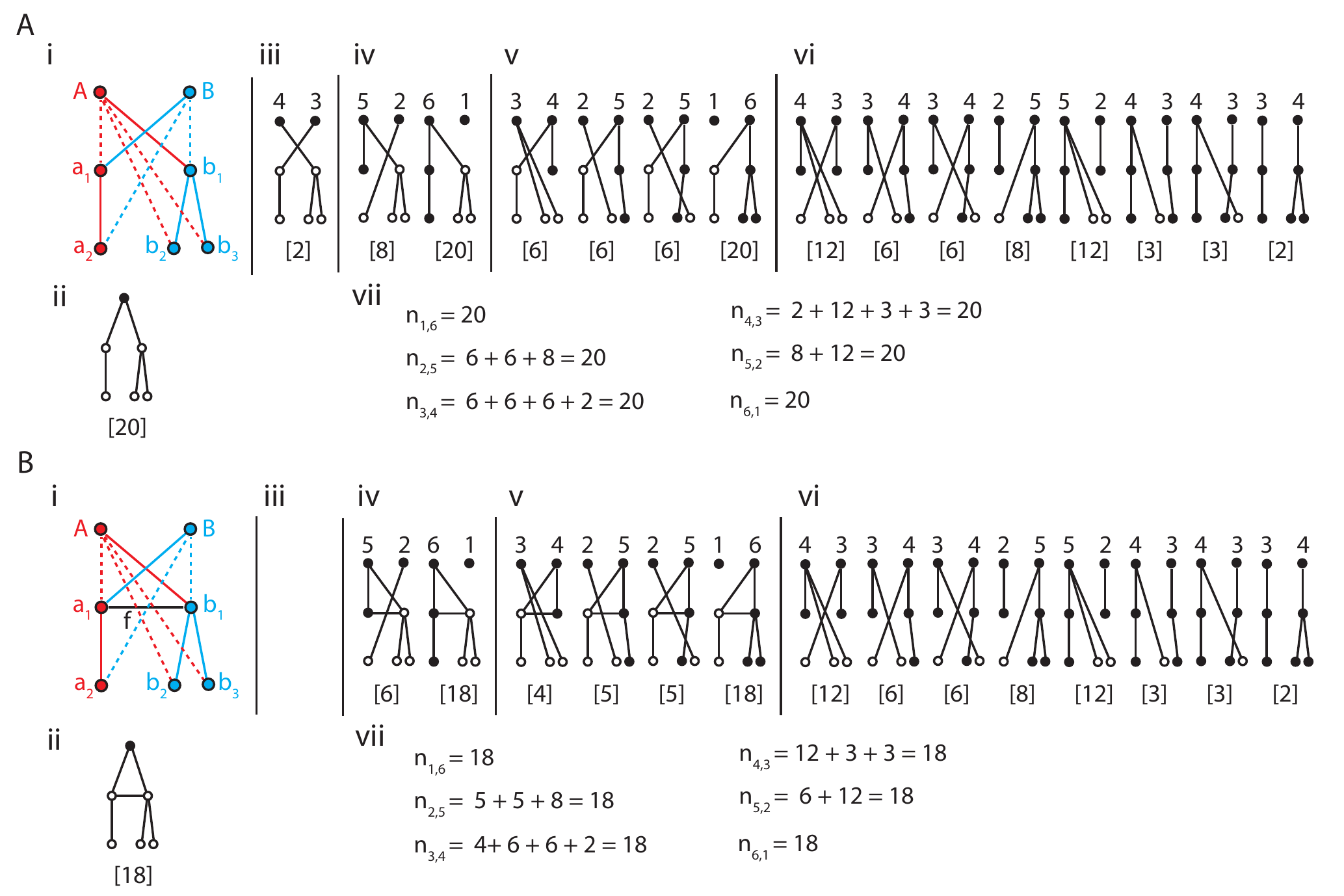}
\caption{Full sets of tree operations. In A) we have a fenceless structure, in B) we have the same structure with a fence $f$. i) The full 2d-trees; blue are type $b$ nodes or edges, red are type $a$ nodes or edges. Solid lines are major edges, dashed lines are minor edges. Black edges are fences. ii) The major graph when nodes $A$ and $B$ are contracted. iii)-vi) Major graphs corresponding to $\beta$-subtrees indicated by blackened nodes. $\beta$-subtrees are partitioned into subgraphs $\tau_A$ connected to node $A$ and $\tau_B$ connected to node $B$. iii) Major graphs when both $\tau_A,\tau_B=\phi$ are empty. iv) Major graphs when $\tau_A=\phi$ and $\tau_B \neq \phi$. v) The trees when $\tau_A \neq \phi$ and $\tau_B=\phi$. vi) Major graphs when $\tau_A,\tau_B \neq \phi$. vii) The total $n_{r,7-r}$ of combinatorial terms of trees with $r$ nodes in component connected to node $A$.}
\label{TD_Tree_Operations}
\end{figure}

Secondly, we generalize the notion of $1$-nodesets.

\begin{definition}
A \emph{$\beta$-subtree} $\tau$ of a $\beta$-tree $T$ is a subset of nodes from $T$ such that:

i) The two root nodes are in $\tau$.

ii) If a node in $\tau$ is the parent of a fence, one of the two daughter nodes bridged by the fence must also be in $\tau$.

iii) If a node is in $\tau$, both parental nodes are also in $\tau$.
\label{Subtree}
\end{definition}

\begin{example} Consider Figure \ref{TD_Tree_Operations}A,B. Here the original $\beta$-trees are in Figures \ref{TD_Tree_Operations}A,Bi. The $\beta$-subtrees are indicated in Figure \ref{TD_Tree_Operations}A,B iii-iv by the solid nodes. Note that the two roots are always in $\tau$. These are the parents of the fence $f$ in Figure \ref{TD_Tree_Operations}B and so in agreement with Definition \ref{Subtree}ii we find that at least one of the two nodes $a_1$ and $b_1$ bridged by $f$ lies in $\tau$.
\end{example}

The $\beta$-subtree $\tau$ of a $\beta$-tree $T$ can be used to define a modified major graph, analogously to the construction of $T_{maj}(E')$ in Lemma \ref{TreeOperations}, as follows:

\begin{definition}
For a $\beta$-tree $T$ and $\beta$-subtree $\tau$, the induced tree $T(\tau)$ is the major graph obtained from $T$ by the following operations.

i) Select all nodes from $T$. 

ii) For any node (that is not a root) in $\tau$ of type $a$ (resp. $b$), select the parental edge of same type $a$ (resp. type $b$).

iii) If any node (that is not a root) of type $a$ (resp. $b$) is immediately adjacent (but not in) $\tau$, select the parental edge of opposite type $b$ (resp. type $a$).

iv) If any node is neither a member of, or immediately adjacent to, $\tau$, select the major parental edge from $T$.

v) If two nodes are connected by a fence in $T$, select the fence if and only if one or both nodes are not in $\tau$.
\label{NewTree}
\end{definition}

Note that this definition differs from Lemma \ref{TreeOperations} in one important way. By Definition \ref{Subtree}, any $\beta$-subtree contains both root nodes. By Definition \ref{NewTree}v, any fence between the two root nodes is not selected in $T(\tau)$. However, we find by Lemma \ref{TreeOperations}ii that $T_{maj}(E')$ will contain a fence between the two roots. We will later see that this difference has an important implication for the calculation of the total number of possible TD-Evolutions.

In order to introduce the main property of $\beta$-trees that will allow us to count TD-Evolutions we need to introduce some notation.
\\
\newline
{\bf Terminology}
\begin{itemize}
    \item For any $2$d-tree $T$ with major graph $T(\tau)$ corresponding to $\beta$-subtree $\tau$:
         \begin{itemize} 
             \item $\overline{T}(\tau)$ is the graph obtained when the two root nodes of $T(\tau)$ are contracted together.
             \item $C(\tau)$ is the product of combinatorial coefficients across nodes and fences of $T(\tau)$ given by Theorem \ref{CountTheorem}.
         \end{itemize}
    \item $\mathcal{S}$ denotes the set of valid $\beta$-subtrees according to Definition \ref{Subtree}.
    \item $\varepsilon$ denotes the trivial $\beta$-subtree containing just the two root nodes.
    \item For any node or fence $x$ in $T$:
         \begin{itemize} 
             \item $N_x(\tau)$ is the number of nodes from $x$ and its descendants in $T$, attached to the $A$ root in $T(\tau)$.
             \item $\tau_x$ denotes the restriction of $\tau$ to $x$ and its descendants.
             \item $\mathcal{S}_x$ denotes the set of possible subsets $\tau_x$.
             \item $C_x(\tau)$ denotes the product of factors of $C(\tau)$ arising from $x$ and its descendants.
             \item $c_x(\tau)$ denotes the single factor associated with node (or fence) $x$.
             \item Terms with an overline added, such as $\overline{c}_x(\tau)$, are the corresponding terms using $\overline{T}(\tau)$ instead of $T(\tau)$.
             \item When terms, such as $c_x(\tau)$ (and $C_x(\tau)$), only depend upon $x$ (and its descendants) in $T(\tau)$, we equivalently use notation $c_x(\tau_x)$ (and $C_x(\tau_x)$).
         \end{itemize} 
   
\end{itemize}

\begin{example} In Figure \ref{TD_Tree_Operations}Bii we have a $\beta$-tree. The graphs in each of Figures \ref{TD_Tree_Operations}Biv-vi are the possible major graphs $T(\tau)$, where each $\beta$-subtree $\tau$ can be identified from the solid nodes. The counts $N_A(\tau)$ can be seen above the $A$ node for each graph. Node $b_1$ in the $\beta$-tree in Figure \ref{TD_Tree_Operations}Bii has two daughter branches with one node each, so we write $c_{b_1}(\varepsilon)={2 \choose 1}=2$. We trivially have $c_{b_2}(\varepsilon)=c_{b_3}(\varepsilon)=1$ for leaf nodes $b_2$ and $b_3$; the descendants of $b_1$, so can also write $C_{b_1}(\varepsilon)=c_{b_1}(\varepsilon)\cdot c_{b_2}(\varepsilon)\cdot c_{b_3}(\varepsilon)=2$.
\end{example}

We are finally in a position to describe the following fundamental result, which will allow us to determined the total number of TD-Evolutions.

\begin{theorem}
Let $T$ be any $\beta$-tree with $N$ nodes, and $\mathcal{S}$ the corresponding set of $\beta$-subtrees. Then for any $r \in \{1,2,...,N-1\}$ we have:

\begin{equation}
\sum\limits_{\{\tau \in \mathcal{S}:N_A(\tau)=r\}}C(\tau)=\overline{C}(\varepsilon)
\label{MainEquation}
\end{equation}

\label{KernelTheorem}
\end{theorem}

\begin{example} We see in Figure \ref{TD_Tree_Operations}Bii the major graph $\overline{T}(\varepsilon)$, where the two root nodes have been contracted together. There are two non-trivial combinatorial terms. One from the fence $f$ below the root, which has two nodes descending the left side, three the right, resulting in combinatorial coefficient $\overline{c}_f(\varepsilon)=({5 \choose 2}-1)=9$, by Theorem \ref{CountTheorem}. The other term comes from the daughter node $b_1$ to the right of the root, which has two daughter branches with one node each, resulting in combinatorial term $\overline{c}_{b_1}(\varepsilon)={2 \choose 1}=2$. All other nodes give coefficient $1$. Together we get the factor $\overline{C}(\varepsilon)=\overline{c}_f(\varepsilon)\cdot \overline{c}_{b_1}(\varepsilon)\cdot1=18$, the number in square brackets given below the graph. Now there are two major graphs $T(\tau_1)$, $T(\tau_2)$ for which $N_A(\tau_1)=N_A(\tau_2)=5$; the first graph in Figure \ref{TD_Tree_Operations}Biv, which has combinatorial term $C(\tau_1)=6$, and the fifth graph in Figure \ref{TD_Tree_Operations}Bvi, which has combinatorial term $C(\tau_2)=12$. These add up to the same value $18$ we found for the graph with contracted roots, agreeing with Equation (\ref{MainEquation}) for $r=5$.
\end{example}

\begin{proof} (\emph{of Theorem \ref{KernelTheorem}})
We prove this by induction on the number of nodes in the $\beta$-tree.
\\
\newline
\emph{\bf Induction Initial Case}
\\

Firstly consider $\beta$-trees with $N=2$ nodes in total. There are two root nodes $A$ and $B$ of type $a$ and $b$, respectively. There are only two possible $\beta$-trees depending on whether $A$ and $B$ are linked by a fence. Now any subtree from $\mathcal{S}$ must contain the two roots by Definition \ref{Subtree}i, so there is only the one subtree $\varepsilon=\{A,B\}$ to consider. If there is a fence between $A$ and $B$ then by Definition \ref{NewTree}v, the fence does not belong to $T(\varepsilon)$. Thus for both cases $T(\varepsilon)$ contains both root nodes and no edges. Now the only value that $r \in 1,...,N-1=1$ can take is $1$ and $N_A(\tau)=r=1$; the number of nodes attached to node $A$ is $1$. Now from Theorem \ref{CountTheorem}, the combinatorial term associated with each node $A$ and $B$ is $1$. Furthermore, the contracted tree $\overline{T}(\varepsilon)$ is a single node, which similarly has a combinatorial term of $1$. We then find that:

\begin{center}
$\sum\limits_{\{\tau \in \mathcal{S}:N_A(\tau)=r\}}C(\tau)=1.1=1=\overline{C}(\varepsilon)$
\end{center}

The result is therefore correct for $\beta$-trees with $N=2$ nodes.
\\
\newline
\emph{\bf Inductive Assumption}
\\

Next we make the inductive hypothesis that the theorem is true for all $\beta$-trees with $N \le K-1$ nodes, for some $K > 2$. We now consider a $\beta$-tree $T$ with $N=K$ nodes. We have two root nodes, $A$ and $B$. The daughter nodes from these two roots may be either type $a$ nodes from root node $B$, type $b$ nodes from root node $A$, or fences descending from both nodes, as portrayed in Figure \ref{TD_Proof_Pic}Ai.

Note also that although the original $\beta$-tree $T$ can have type $a$ nodes with major parent $A$, or type $b$ nodes with major parent $B$, they can effectively be assumed to have opposite parentage. More specifically, if we have a daughter node $x_a$ of type $a$ descending from either root node, then when any $\beta$-subtree $\tau$ not including $x_a$ (such as $\varepsilon$) is used, we find $x_a$ is attached to $B$ in $T(\tau)$ by Definition \ref{NewTree}iii. Also, for any $\beta$-subtree $\tau$ including $x_a$ (such as the entire nodeset from $T$), we find $x_a$ is attached to $A$ in $T(\tau)$ by Definition \ref{NewTree}ii. When the two roots are contracted in $\overline{T}(\varepsilon)$ we find node $x_a$ attached to the single root. Thus the choice of which root to use as the major parent of $x_a$ has no affect on the validity of Equation (\ref{MainEquation}) and we take the root $B$ as stated. The argument for daughter node $x_b$ is similar. This is equivalent to assuming $T=T(\varepsilon)$.

Although there may be any number of these type of branches descending from the roots, for the sake a simpler exposition we provide the proof just for the four branches drawn in Figure \ref{TD_Proof_Pic}Ai. The generalization is relatively straightforward (see comment at end of proof). 

Now we require a sum over $\beta$-subtrees $\tau$ such that $N_A(\tau)=r$. That is, we require $r$ nodes in the component of $T(\tau)$ attached to root node $A$. We suppose that in the original tree $T$ (see Figure \ref{TD_Proof_Pic}Ai) there are $n_a$, $n_b$ and $n_f$ nodes contained in each of the two branches containing nodes $a$ and $b$, and the two branches containing the fence $f$, respectively, where $n_f=n_{a'}+n_{b'}$. We suppose that there are subsequently $r_a$, $r_b$ and $r_f=r_{a'}+r_{b'}$ of these nodes attached to $A$ in major tree $T(\tau)$, such as in Figure \ref{TD_Proof_Pic}Aii. The count $r$ includes node $A$ so we require $r_a+r_b+r_f=r-1$.

Now $C(\tau)$ is a product of terms across the nodes and fences, which can be split into $A$, $B$, the nodes in the two branches containing $a$ and $b$, and the two branches bridged by fence $f$. Recalling the terminology introduced above, we can split the combinatorial term for major tree $T(\tau)$ as:

\begin{center}
$C(\tau)=c_A(\tau)\cdot c_B(\tau)\cdot C_a(\tau_a)\cdot C_b(\tau_b)\cdot C_{f}(\{\tau_{a'},\tau_{b'}\})$
\end{center}

Note that $\tau_a$, $\tau_b$, $\tau_{a'}$ and $\tau_{b'}$ are the subsets of $\tau$ the include nodes $a$, $b$, $a'$ and $b'$ and their descendants, respectively. Thus when we have trivial $\beta$-subtree $\tau=\varepsilon$, we find these subsets are empty; $\tau_a=\tau_b=\tau_{a'}=\tau_{b'}=\phi$.

The left hand side of Equation (\ref{MainEquation}) can then be split into sums across the four branches as follows:
\begin{equation}
\sum\limits_{\substack{\{\tau \in \mathcal{S}:\\n_A(\tau)=r\}}}C(\tau) =  \sum\limits_{\substack{\{r_a,r_b,r_f:\\r_{a}+r_{b}+r_f\\=r-1\}}} c_A(\tau)c_B(\tau)\cdot
\sum\limits_{\substack{\{\tau_a \in \mathcal{S}_a:\\N_a(\tau_a)=r_a\}}}C_a(\tau_a)\cdot \sum\limits_{\substack{\{\tau_b \in \mathcal{S}_b:\\N_b(\tau_b)=r_b\}}}C_b(\tau_b)\cdot\\
\sum\limits_{\substack{\{\tau_{a'} \in \mathcal{S}_{a'},\tau_{b'} \in \mathcal{S}_{b'}:\\N_f(\{\tau_{a'},\tau_{b'}\})=r_f\\ \sim (\tau_{a'},\tau_{b'}=\phi)\}}}C_{f}(\{\tau_{a'},\tau_{b'}\})
\label{MessyEquation}
\end{equation}

Now the sum is restricted to $\beta$-subtrees with $r-1$ nodes in the branches descending from $A$. We have three branches from node $A$ with node counts $r_a$, $r_b$ and $r_f=r_{a'}+r_{b'}$, where the two branches containing $a'$ and $b'$ are treated as a single branch in accordance with Theorem \ref{CountTheorem}. Thus we find that we associate node $A$ with the multinomial coefficient:
 
\begin{equation}
c_A(\tau)={r-1 \choose r_{a},r_{b},r_f}
\label{c_A}
\end{equation}

We similarly find we have:

\begin{equation}
c_B(\tau)={K-r-1 \choose n_a-r_a,n_b-r_b,n_f-r_f}
\label{c_B}
\end{equation}

We thus have expressions for two terms in Equation (\ref{MessyEquation}). To calculate the remaining terms we show that each branch corresponds to a smaller $\beta$-tree ($<K-1$ nodes) which will enable us to use the inductive hypothesis. We have three cases to consider.
\\
\newline
\emph{\bf Case I: Dealing with Type $a$ Branches}
\\

Instead of the full $\beta$-tree $T$ (represented in Figure \ref{TD_Proof_Pic}Ai), consider the $\beta$-tree in Figure \ref{TD_Proof_Pic}Bi which we obtain by removing all branches except the branch containing $a$, removing edge $B \rightarrow a$ and contracting nodes $A$ to $a$ together. We call the resulting $\beta$-tree $T'$. The corresponding major graphs for $T$ and $T'$ are represented in Figures \ref{TD_Proof_Pic}A,Bii. 

Now every $\beta$-subtree $\tau'$ of $T'$ can be written as $\tau'=\{\tau_a,B\}$ for some $\tau_a \in \mathcal{S}_a$. This correspondence applies for every $\tau_a \neq \phi$. For this single case, $\tau_a = \phi$ we find the root $a$ for $T'$ is missing from $\{\tau,B\}$, and we do not have a valid $\beta$-subtree of $T'$. We thus treat the two cases of $\tau_a=\phi$ and $\tau_a \neq \phi$ separately. 

{\bf Case Ia} ($\tau_a=\phi$). Now the major branch of $2$d-tree $T$ containing $a$ is unmodified in $T(\tau)$. Thus all $n_a$ nodes in the branch containing the $a$ node are in one component connected to the root $B$ node, and so $N_a(\tau_a)=r_a=0$. Now for any $\tau_a \neq \phi$, node $a$ is attached to the root $A$ in $T(\tau)$ by Definition \ref{NewTree}ii, so $r_a>0$. Thus the only case with $r_a=0$ is $\tau_a=\phi$. For this case we note that $C_a(\tau_a)=\overline{C}_a(\varepsilon)$, and the following equation holds true.

\begin{equation}
\sum\limits_{\substack{\{\tau_a \in \mathcal{S}_a:\\N_a(\tau_a)=r_a\}}}C_a(\tau_a)=\overline{C}_a(\varepsilon)
\label{Ca}
\end{equation}

{\bf Case Ib} ($\tau_a \neq \phi$) We next verify Equation (\ref{Ca}) for values $r_a \neq 0$.

Now for $\tau_a \neq \phi$ we have well defined $\beta$-subtrees $\tau'=\{\tau_a,B\}$. Furthermore, the descendants of root $a$ in the major graph $T'(\tau')$ match the descendants of node $a$ in major graph $T(\tau)$, and we find that the combinatorial term associated to tree $T'(\tau')$ will be precisely $C_a(\tau_a)$. Noting that $T'$ has at least one less node than $T$, we can apply the inductive hypothesis using Equation (\ref{MainEquation}) to $T'$ and hence derive Equation (\ref{Ca}) for the remaining cases where $r_a>0$.
\\
\newline
\emph{\bf Case II: Dealing with Type $b$ Branches}
\\

By a symmetric argument on the branch with node $b$ we obtain an analogous equation of the form:

\begin{equation}
\sum\limits_{\substack{\{\tau_a \in \mathcal{S}_b:\\N_b(\tau_b)=r_b\}}}C_b(\tau_b)=\overline{C}_b(\varepsilon)
\label{Cb}
\end{equation}
\emph{\bf Case III: Dealing with Daughter Fences}
\\

We are interested in the remaining combinatorial term $C_f(\{\tau_{a'},\tau_{b'}\})$ from Equation (\ref{MessyEquation}) that we have yet to examine. This corresponds to the two branches containing nodes $a'$ and $b'$, and the fence $f$ between them. We thus define $\beta$-tree $T'$ as the restriction of $T$ to these two branches (see Figure \ref{TD_Proof_Pic}Ci). If we also remove the fence $f$, we get the $\beta$-tree in Figure \ref{TD_Proof_Pic}Di, which we call $T''$. We use terms, such as $C'$ and $C''$ for example, to refer to combinatorial terms associated to $T'$ and $T''$.

The reason for doing this is because the combinatorial terms of the two sets of induced major graphs are closely related, which we will exploit. For example, in Figure \ref{TD_Tree_Operations}A we see a $\beta$-tree with a fence and in \ref{TD_Tree_Operations}B we see the same $\beta$-tree with the fence removed. The combinatorial terms (in square brackets below each graph in Figures \ref{TD_Tree_Operations}A,Biii-vi) are identical in all cases except when either $\tau_{a'}=\phi$ or $\tau_{b'}=\phi$ is empty.

First consider $T''$. For the graph in Figure \ref{TD_Proof_Pic}Di we have $n_{b'}$ nodes descending from node $A$ and $n_{a'}$ from node $B$. Now because there is no fence $f$ present in $T''$ we have two separate branches; one from root $A$ down the branch containing node $b'$, the other from root $B$ down the branch containing node $a'$. We can then apply the same methods as Cases I and II above to conclude Equation (\ref{MainEquation}) is valid for $T''$ (Figure \ref{TD_Proof_Pic}D). This gives us:

\begin{equation}
\sum\limits_{\substack{\{\tau_{a'} \in \mathcal{S}_{a'} \\ \tau_{b'} \in \mathcal{S}_{b'}\\ N_f(\{\tau_{a'},\tau_{b'}\})=r_f\}}}C''(\tau_{a'},\tau_{b'})  = \overline{C}_{a'}''(\varepsilon)\cdot \overline{C}_{b'}''(\varepsilon)\cdot {n_{a'}+n_{b'} \choose n_{a'}}= \overline{C}_{a'}(\varepsilon)\cdot \overline{C}_{b'}(\varepsilon)\cdot {n_{a'}+n_{b'} \choose n_{a'}}
\label{Equation4Proof}
\end{equation}

Here we have used the fact that the components $\overline{C}_{a'}''(\varepsilon)$ and $\overline{C}_{b'}''(\varepsilon)$ derived from the $\beta$-tree $T''$ (corresponding to the two triangles in Figure \ref{TD_Proof_Pic}Diii) are identical to the components $\overline{C}_{a'}(\varepsilon)$ and $\overline{C}_{b'}(\varepsilon)$ derived from the original $\beta$-tree $T$ (see Figure \ref{TD_Proof_Pic}Aiii). The combinatorial term ${n_{a'}+n_{b'} \choose n_{a'}}$ arises because when the two root nodes are contracted together the single resulting node has two descending branches containing $n_{a'}$ and $n_{b'}$ nodes (see Figure \ref{TD_Proof_Pic}Diii), and we then apply Theorem \ref{CountTheorem}. 

Now we want the corresponding sum to Equation (\ref{Equation4Proof}) for tree $T'$. We have four cases to consider depending on whether $\tau_{a'}$ or $\tau_{b'}$ are empty.
\\
\newline
{\bf Case IIIi} ($\tau_{a'},\tau_{b'}\neq\phi$) Now if both subsets $\tau_{a'}$ and $\tau_{b'}$ are non-empty, we find from Lemma \ref{TreeOperations}v that node $a'$ is attached to root $A$ and node $b'$ is attached to root $B$, and the fence is not part of $T'$. This results in the identical situation to $T''$, where there was no fence $f$ in the first place. We then find that:

\begin{equation}
\sum\limits_{\substack{\{\tau_{a'} \neq \phi \\ \tau_{b'} \neq \phi\\ N_f(\{\tau_{a'},\tau_{b'}\})=r_f\}}}C'(\{\tau_{a'},\tau_{b'}\}) = \sum\limits_{\substack{\{\tau_{a'} \neq \phi \\ \tau_{b'} \neq \phi\\ N_f(\{\tau_{a'},\tau_{b'}\})=r_f\}}}C''(\{\tau_{a'},\tau_{b'}\})
\label{CaseIIIi}
\end{equation}

An example of this can be seen in Figure \ref{TD_Tree_Operations}A,Bvi, where the combinatorial terms (in square brackets) are equal between the two groups.
\\

\begin{figure}[t!]
\centering
\includegraphics[width=140mm]{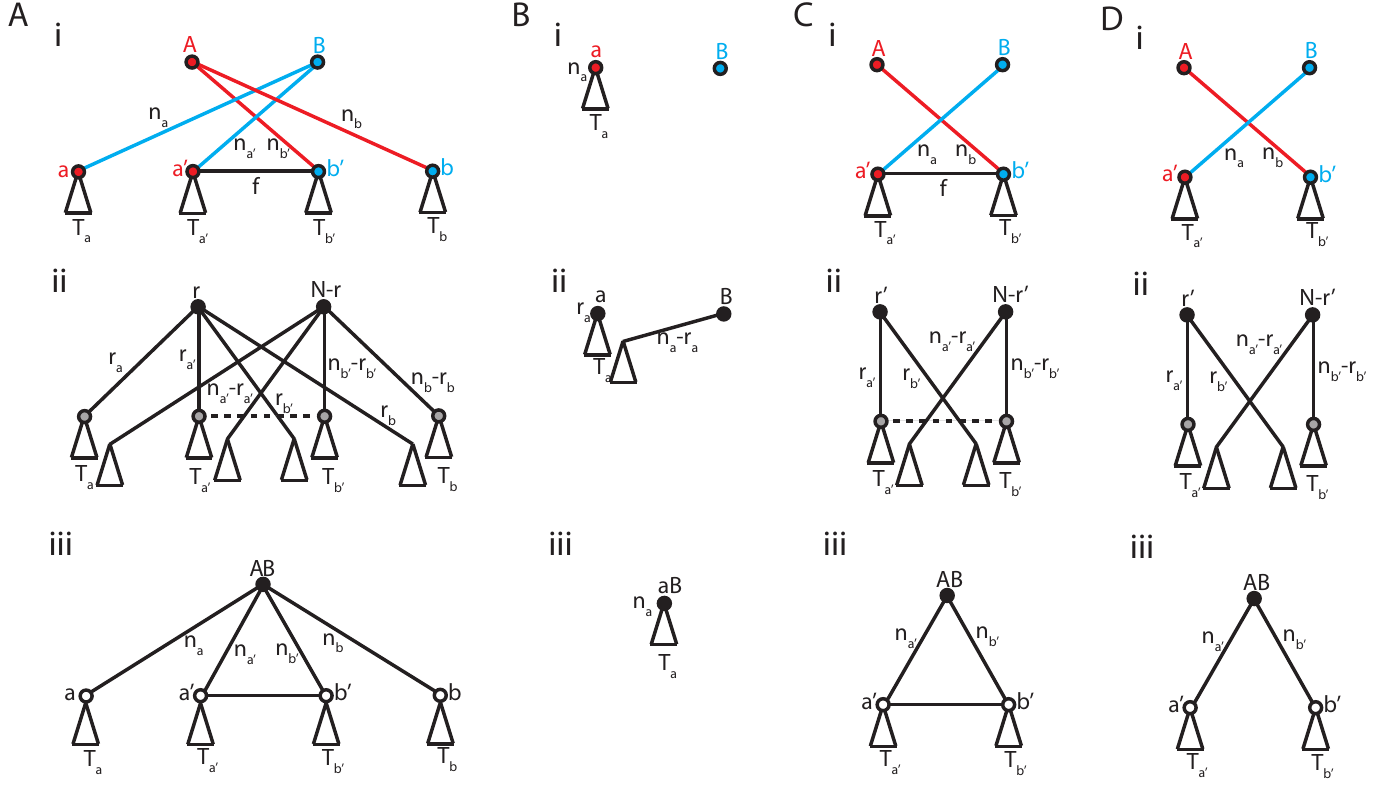}
\caption{A) The general form of a 2d-tree. Triangles indicate a 2d-tree substructure. Dashed lines indicate possible presence of a fence. B) Reduction to a single branch. C) Reduction to a descending fence. D) The graphs of C with the fence removed. i) 2d-trees. ii) Trees $T(\tau)$. iii) Trees $\overline{T(\tau)}$ after root node contraction.}
\label{TD_Proof_Pic}
\end{figure}

Now from Equation (\ref{MessyEquation}) we are interested in the subsets $\tau_{a'}$ and $\tau_{b'}$ such that the number of nodes either bridged by, or descending from, fence $f$ is equal to $r_f=r_{a'}+r_{b'}$ for some value $r_f$. We have just seen that when $\tau_{a'}$ and $\tau_{b'}$ are both non-empty, the two sums in Equation (\ref{Ca}) corresponding to trees $T'$ and $T''$ are equal for all values of $r_f$. For the remaining three cases, where at lease one of $\tau_{a'}$ and $\tau_{b'}$ is empty, we will see that there is a constant difference between the sums arising from trees $T'$ and $T''$. Furthermore, the value $r_f$ will be seen to arise in exactly one of these three cases.
\\
\newline
{\bf Case III ii} ($\tau_{a'},\tau_{b'} = \phi$ and $r_f=n_{b'}$) The $\beta$-subtree $\tau=\varepsilon$ is trivial and there are no changes to the major graph. We then find there are $n_{b'}$ nodes present in the branch descending from $A$ in $T''(\tau)$. We thus find this case applies if $r_f=n_{b'}$. Now this situation does not apply to $T'$ (when the fence $f$ is present). The parental nodes of $f$ (the roots in this case) lie in $\tau=\varepsilon$, so one of the two nodes bridged by $f$ must lie in the $\beta$-subtree $\tau$ by Definition \ref{Subtree}ii. We thus find that although $\varepsilon \in \mathcal{S}''$ in a valid $\beta$-subtree for $T''$, $\varepsilon \not\in \mathcal{S}'$ is not a valid $\beta$-subtree for $T'$ (Figure \ref{TD_Tree_Operations}A,Biii only has a contribution for the fenceless graph, for example). For $T''$, the trivial conditions $\tau_{a'},\tau_{b'} = \phi$ result in a single induced major tree (corresponding to Figure \ref{TD_Proof_Pic}Di), with major edges that match the original $\beta$-tree $T$, and we obtain combinatorial term $C''_f(\varepsilon)=\overline{C}_{a'}(\varepsilon)\cdot \overline{C}_{b'}(\varepsilon)$. We thus find:

\begin{equation}
\sum\limits_{\substack{\{\tau_{a'} \neq \phi \\ \tau_{b'} \neq \phi\\ N_f(\{\tau_{a'},\tau_{b'}\})=r_f\}}}C'(\{\tau_{a'},\tau_{b'}\}) = \sum\limits_{\substack{\{\tau_{a'} \neq \phi \\ \tau_{b'} \neq \phi\\ N_f(\{\tau_{a'},\tau_{b'}\})=r_f\}}}C''(\tau_{a'},\tau_{b'})-\overline{C}_{a'}(\varepsilon)\cdot \overline{C}_{b'}(\varepsilon)
\label{CaseIIIii}
\end{equation}

{\bf Case III iii} ($\tau_{a'} = \phi$, $\tau_{b'} \neq \phi$ and $r_f<n_{b'}$) Now  $\tau_{a'}$ is trivial, so the arm descending from $B$ containing node $a'$ is unchanged from the original major graph for both trees $T'$ and $T''$, and all $n_{a'}$ nodes remain in the component of the major graph containing $B$ ($r_{a'}=0$). After the changes induced by $\tau_{b'}$, the other arm splits with $r_{b'}$ nodes belonging to the component of the major tree containing $A$,  and $n_{b'}-r_{b'}$ nodes belonging to the component containing $B$. Thus in total there are $r_f=r_{b'}$ nodes from the original two branches that end up in the component of the major graph containing $A$, for some $r_{b'} \in \{1,2,...,n_{b'}-1\}$. This case will thus apply provided $r_f<n_{b'}$. Now the combinatorial term from the unmodified $a'$ branch matches those from the original $\beta$-tree; $\overline{C}_{a'}(\varepsilon)$. Now in the tree $T''$ (without the fence $f$) the branch containing node $b'$ can be treated with the inductive hypothesis, like Case II above, and we find that:

\begin{equation}
\sum\limits_{\substack{\{\tau_{a'} = \phi \\ \tau_{b'} \neq \phi\\ N_f(\{\tau_{a'},\tau_{b'}\})=r_f\}}}C''(\tau_{a'},\tau_{b'}) = \overline{C}_{a'}(\varepsilon)\cdot \sum\limits_{\substack{\{\tau_{b'} \neq \phi\\ N_{b'}(\tau_{b'})=r_f\}}}C''_{b'}(\tau_{b'})\cdot{n_{b'}-r_{b'}+n_{a'} \choose n_{a'}}=\overline{C}_{a'}(\varepsilon)\cdot \overline{C}_{b'}(\varepsilon)\cdot {n_{b'}-r_{b'}+n_{a'} \choose n_{a'}}
\label{StopGap}
\end{equation}

Here we pick up a combinatorial factor ${n_{b'}-r_{b'}+n_{a'} \choose n_{a'}}$ from the two branches descending from node $B$. Now for the tree $T'$ (with fence $f$), the only combinatorial factor that differs between any pair of induced major trees $T'(\tau)$ and $T''(\tau)$, is the combinatorial term from node $B$ in $T''(\tau)$, which becomes the fence factor for $f$; ${n_{b'}-r_{b'}+n_{a'} \choose n_{a'}}-1$ in $T'(\tau)$. That is:

\begin{center}
$\frac{C''(\{\tau_{a'},\tau_{b'}\})}{{n_{b'}-r_{b'}+n_{a'} \choose n_{a'}}}=\frac{C'(\{\tau_{a'},\tau_{b'}\})}{{n_{b'}-r_{b'}+n_{a'} \choose n_{a'}}-1}$.
\end{center}

Substituting this into Equation (\ref{StopGap}) gives us:

\begin{equation}
\sum\limits_{\substack{\{\tau_{a'} = \phi \\ \tau_{b'} \neq \phi\\ N_f(\{\tau_{a'},\tau_{b'}\})=r_f\}}}C'(\{\tau_{a'},\tau_{b'}\}) =\overline{C}_{a'}(\varepsilon)\cdot \overline{C}_{b'}(\varepsilon)\cdot \left({n_{b'}-r_{b'}+n_{a'} \choose n_{a'}}-1\right)
\label{StopGap2}
\end{equation}

Then subtracting Equation (\ref{StopGap2}) from Equation (\ref{StopGap}) reveals the same constant difference observed in the previous case:

\begin{equation}
\sum\limits_{\substack{\{\tau_{a'} = \phi \\ \tau_{b'} \neq \phi\\ N_f(\{\tau_{a'},\tau_{b'}\})=r_f\}}}C'(\{\tau_{a'},\tau_{b'}\}) = \sum\limits_{\substack{\{\tau_{a'} = \phi \\ \tau_{b'} \neq \phi\\ N_A=r_f+1\}}}C''(\{\tau_{a'},\tau_{b'}\})-\overline{C}_{a'}(\varepsilon)\cdot \overline{C}_{b'}(\varepsilon)
\label{CaseIIIiii}
\end{equation}

{\bf Case III iv}: ($\tau_{b'} = \phi$, $\tau_{a'} \neq \phi$ and $r_f>n_{b'}$) The argument is analogous to Case III ii and the same difference is obtained where we find:

\begin{equation}
\sum\limits_{\substack{\{\tau_{a'} \neq \phi \\ \tau_{b'} = \phi\\ N_f(\{\tau_{a'},\tau_{b'}\})=r_f\}}}C'(\{\tau_{a'},\tau_{b'}\}) = \sum\limits_{\substack{\{\tau_{a'} \neq \phi \\ \tau_{b'} = \phi\\ N_A=r_f+1\}}}C''(\{\tau_{a'},\tau_{b'}\})-\overline{C}_{a'}(\varepsilon)\cdot \overline{C}_{b'}(\varepsilon)
\label{CaseIIIiv}
\end{equation}

Thus in all three cases (III i-iii) the difference between the tree with and without the fence is $\overline{C}_{a'}(\varepsilon)\cdot \overline{C}_{b'}(\varepsilon)$. Furthermore, for any single value of $r_f$, one of these three cases is applicable. We thus find, using Equations (\ref{CaseIIIii}), (\ref{CaseIIIiii}) and (\ref{CaseIIIiv}) with (\ref{CaseIIIi}) that:

\begin{center}
$\sum\limits_{\substack{\{\tau_{a'} \in \mathcal{S}_{a'} \\ \tau_{b'} \in \mathcal{S}_{b'}\\ N_f(\{\tau_{a'},\tau_{b'}\})=r_f\}}}C'(\{\tau_{a'},\tau_{b'}\}) = \sum\limits_{\substack{\{\tau_{a'} \in \mathcal{S}_{a'} \\ \tau_{b'} \in \mathcal{S}_{b'}\\ N_f(\{\tau_{a'},\tau_{b'}\})=r_f\}}}C''(\{\tau_{a'},\tau_{b'}\})-\overline{C}_{a'}(\varepsilon)\cdot \overline{C}_{b'}(\varepsilon)$
\end{center}

Then substituting this into Equation (\ref{Equation4Proof}) gives us:

\begin{equation}
\begin{array}{l l}
\sum\limits_{\substack{\{\tau_{a'} \in \mathcal{S}_{a'} \\ \tau_{b'} \in \mathcal{S}_{b'}\\ N_f(\{\tau_{a'},\tau_{b'}\})=r_f\}}}C'(\{\tau_{a'},\tau_{b'}\}) & = \overline{C}_{a'}(\varepsilon)\cdot \overline{C}_{b'}(\varepsilon)\cdot {n_{a'}+n_{b'} \choose n_{a'}}-\overline{C}_{a'}(\varepsilon)\cdot \overline{C}_{b'}(\varepsilon)\\
& = \overline{C}_{a'}(\varepsilon)\cdot \overline{C}_{b'}(\varepsilon)\cdot ({n_{a'}+n_{b'} \choose n_{a'}}-1)
\end{array}
\label{FenceComponent}
\end{equation}

Now $C'(\{\tau_{a'},\tau_{b'}\})$ matches the combinatorial term from the fence and its descendants, $C_f(\{\tau_{a'},\tau_{b'}\})$. Furthermore $\overline{C}_{a'}(\varepsilon)$, $\overline{C}_{b'}(\varepsilon)$ and $({n_{a'}+n_{b'} \choose n_{a'}}-1)$ match the terms in the graph $\overline{T}$ obtained from the branch containing node $a'$, the branch containing node $b'$, and fence $f$ (by Theorem \ref{CountTheorem}), and so equal $\overline{C}'(\varepsilon)$. Thus we find that:

\begin{equation}
\sum\limits_{\substack{\{\tau_{a'} \in \mathcal{S}_a \\ \tau_{b'} \in \mathcal{S}_{b'}\\ N_f(\{\tau_{a'},\tau_{b'}\})=r_f\}}}C_f(\{\tau_{a'},\tau_{b'}\}) = \overline{C}_f(\varepsilon)
\label{Fence}
\end{equation}
\\
\newline
\emph{\bf Completing the Induction}
\\

Thus finally substituting Equations (\ref{c_A}), (\ref{c_B}), (\ref{Ca}), (\ref{Cb}) and (\ref{Fence}) into Equation (\ref{MessyEquation}) we find that we have:

\begin{center}
$\sum\limits_{\{\tau \in \mathcal{S}:N_A(\tau)=r\}}C(\tau)= \overline{C}_a(\varepsilon)\cdot\overline{C}_b(\varepsilon)\cdot\overline{C}_f(\varepsilon) \cdot\sum\limits_{\substack{\{r_a,r_b,r_f:\\r_{a}+r_{b}+r_f\\=r-1\}}}{r-1 \choose r_a,r_b,r_f}{K-r-1 \choose n_a-r_a,n_b-r_b,n_f-r_f}$
\end{center}

Then applying the multinomial version of the Vandermonde identity \cite{Zeng} results in:

\begin{center}
$\sum\limits_{\{\tau \in \mathcal{S}:N_A(\tau)=r\}}C(\tau)= \overline{C}_a(\varepsilon)\cdot\overline{C}_b(\varepsilon)\cdot\overline{C}_f(\varepsilon) \cdot{K-2 \choose n_a,n_b,n_f}$
\end{center}

However, the multinomial coefficient is identical to the combinatorial term we get if nodes $A$ and $B$ are contracted to a single root. In Figure \ref{TD_Proof_Pic}Aiii we see two branches containing nodes $n_a$ and $n_b$, these have combinatorial terms equal to $\overline{C}_a(\varepsilon)$ and $\overline{C}_b(\varepsilon)$. We also have fence $f$ bridging the two branches containing nodes $n_{a'}$ and $n_{b'}$. Application of Theorem \ref{CountTheorem} to $f$ for the root contracted graph $\overline{T}(\varepsilon)$ (given in Figure \ref{TD_Proof_Pic}Aiii) returns precisely the term ${K-2 \choose n_a,n_b,n_f}$. We thus find that we have all the coefficients of $C(\overline{T})$ and Equation (\ref{MainEquation}) is obtained.

If we have more than one branch descending from the root nodes, the only change to the argument above is that we sum over a greater number of $r_i$ values. The Vandermonde identity still applies and the same result is obtained.
\end{proof}


\subsection{Proving the Main Result} 

Finally, we can use this inductive relationship to determine the number of different evolutions that arise from a TD process, and prove our main result.

\begin{proof} (Proof of Theorem \ref{EvoCounter})
Let $E \in \mathcal{W}_{n-1}$ be a word evolution on $n-1$ TDs with 2d-tree $T$, and $\mathcal{E}(E) \subset \mathcal{W}_{n}$ the corresponding subset of induced evolutions. Let $\mathcal{N}(E)$ and $\mathcal{N}(\mathcal{E}(E))$ denote the number of TD-Evolutions corresponding to word evolution $E$, and set of induced word evolutions $\mathcal{E}(E)$, respectively. Now for every induced evolution $E'$ we know that there corresponds a $1$-nodeset $\tau$ such that the major graph $T_{maj}(E')$ corresponding to $E'$ is obtained from Lemma \ref{TreeOperations}. By Definition \ref{Subtree}, $\tau$ is also a $\beta$-subtree, and we have an induced major graph $T(\tau)$. We also know from Theorem \ref{KernelTheorem} that for any induced major graph $T(\tau)$ we can sum $C(\tau)$ over the $\beta$-subtrees $\tau$ to obtain $\overline{C}(\varepsilon)$. 

Now the difference between $T(\tau)$ and the major graph $T_{maj}(E)$ is that the latter has a fence between the two root nodes, this difference arising from Lemma \ref{TreeOperations}vi.

Now $E'$ is a word evolution on $n$ TDs so $T_{maj}(\tau)$ has $2n$ nodes. Furthermore, $T_{maj}(E')$ can have $r$ nodes at the type $a$ root, for some $r=\{1,2,...,2n-1\}$, along with $2n-r$ nodes at the type $b$ root. Then given the extra fence between the roots, by Theorem \ref{CountTheorem}, the number of TD-Evolutions associated with $T_{maj}(E')$ is given by $({2n \choose r}-1)C(\tau)$. Then, using Theorem \ref{KernelTheorem}, the total number of TD-Evolutions induced by $E$ is given by:

\begin{center}
$\begin{array}{c c l}
\mathcal{N}(\mathcal{E}(E))=\sum\limits_{r=1}^{2n-1}\sum\limits_{\{\tau \in \mathcal{S}:N_A(\tau)=r\}}({2n \choose r}-1)C(\tau) & = & \sum\limits_{r=1}^{2n-1}({2n \choose r}-1)\cdot\sum\limits_{\{\tau \in \mathcal{S}:N_A(\tau)=r\}}C(\tau)\\
& = & \sum\limits_{r=1}^{2n-1}({2n \choose r}-1)\cdot \overline{C}(\varepsilon)\\
& = & \overline{C}(\varepsilon)\cdot ((\sum\limits_{r=0}^{2n}{2n \choose r}-2)-(2n-1))\\
& = & \overline{C}(\varepsilon).(2^{2n}-(2n+1))
\end{array}$
\end{center}

Now by Theorem \ref{CountTheorem} $\overline{C}(\varepsilon)$ is the number of TD-Evolutions associated with word evolution $E$. Thus we have:

\begin{equation}
\mathcal{N}(\mathcal{E}(E))=\mathcal{N}(E)\cdot(4^{n}-(2n+1))
\label{MainIteration}
\end{equation}

Furthermore, by Corollary \ref{DisjointUnion}, the set of induced evolutions from $\mathcal{W}_{n-1}$ gives rise to a disjoint union of $\mathcal{W}_{n}$. Thus summing Equation (\ref{MainIteration}) across all $E \in \mathcal{W}_{n-1}$ gives $\mathcal{N}_n=\mathcal{N}_{n-1}\cdot(4^{n}-(2n+1))$. Starting a recursion from the single TD-Evolution with $\mathcal{N}_1=1$ then proves the theorem.
\end{proof}


\section{Conclusions}

We have seen in Table \ref{TableCount} from our main result that the number of different evolutions increases with uncompromising velocity. This means that beyond five or six tandem duplications it is at present unrealistic to attempt to computationally explore this space in its entirety. This makes it difficult to compare any observations to the set of possibilities. This is further compounded by the non-uniqueness of copy number vectors, which grow at a far slower rate than the number of evolutions. This means that even if the precise copy number vector is known, it will correspond to a multitude of evolutions, all of which explain the data equally well. One could attempt to apply the type of analyses of \cite{Kinsella} to TDs; it is an open problem to determine the number of copy number vectors, or even an efficient algorithm to determine whether a copy number vector can arise from a process of tandem duplication under the assumption of unique breakpoint use. We see from Table \ref{TableCount} that including the somatic connectivity information of TD-Graphs improves the situation and more evolutions can be distinguished, however, we still have a degeneracy and the underlying evolution cannot necessarily be identified. The development of suitable combinatoric approaches to count the number of possible TD-Graphs (rather than the implemented brute force approach of a computer) also remains an open problem.

The methods utilized in this work largely parallel those used to examine breakage fusion cycles \cite{Greenman1}. These are a distinct form of rearrangement which suggest there may be a more general space in which rearrangements operate and these methods apply. Given that tandem duplication and breakage fusion cycles are leading candidate rearrangements in the formation of large scale copy number increases such as those found in amplicons in cancer, a generalization of these methods to the combined space of these rearrangement processes may help to better understand their evolution.

In this study we have treated the process in a strictly discrete manner. However, one could consider TD as a continuous process on the real line (or stretch of DNA) and investigate the relative likelihoods of different structures arising, as has been done with breakage fusion bridge cycles in \cite{Greenman1}.

The methods above and in \cite{Greenman1} can be viewed as mathematical operations on the real line. It would seem plausible that other duplication mechanisms beyond those found in biological rearrangements would yield to similar analyses, which may shed light on the applicability of these methods which link combinatorics, general automaton on symbolic algebra and duplicating mappings on intervals.

\footnotesize{

}


\begin{thebibliography}{99}
\bibitem{Ohno} Ohno S., 1970, \emph{Evolution by gene duplication}, Springer-Verlag.
\bibitem{Nye} Nye TMW, 2009, Modelling the evolution of multi-gene families, \emph{Stat Methods Med Res}, {\bf 18(5)}, 487-504.
\bibitem{McBride} McBride DJ, Etemadmoghadam D, Cooke SL, Alsop K, George J, Butler A, Cho J, Galappaththige D, Greenman CD, Howarth KD, Lau KW, Ng CK, Raine K, Teague J, Wedge DC, Australian Ovarian Cancer Study Group, Caubit X, Stratton MR, Brenton JD, Campbell PJ, Futreal PA, Bowtell DDL, (2012), Tandem duplication of chromosomal segments is common in ovarian and breast cancer genomes, \emph{J. Pathology}, {\bf 227(4)}, 446 - 455.
\bibitem{Raphael1} Raphael BJ, Pevzner PA. 2004. Reconstructing tumor amplisomes.
\emph{Bioinformatics} {\bf 20}: i265–i273.
\bibitem{Raphael2} Raphael BJ, Volik S, Collins C, Pevzner PA. 2003. Reconstructing tumor genome architectures. \emph{Bioinformatics} {\bf 19}: ii162–ii171.
\bibitem{Zhang} Zhang C, Leibowitz ML and Pellman D, 2013, Chromothripsis and beyond: rapid genome evolution from complex chromosomal rearrangements, \emph{Genes Dev.}, {\bf 27}, 2513-2530.
\bibitem{Gascuel} Gascuel O, Hendy MD, Jean-Marie A, McLachlan R, 2003, The combinatorics of tandem duplication trees, \emph{Syst Biol.}, {\bf 52(1)}, 110-8.
\bibitem{Yang} Yang J and Zhang L, On Counting Tandem Duplication Trees, 2004, \emph{Mol. Biol. Evol.}, {\bf 21(6)}, 1160-1163.
\bibitem{Bertrand} Bertrand D, Lajoie M, and El-Mabrouk N, 2008, Inferring Ancestral Gene Orders for a Family of Tandemly Arrayed Genes, \emph{J. of Comp. Biology}, {\bf 15(8)}, 1063-1077.
\bibitem{Greenman2} Greenman CD, Pleasance ED, Newman S, Yang F, Fu B, Nik-Zainal S, Jones D, Lau KW, Carter N, Edwards PA, Futreal PA, Stratton MR, Campbell PJ (2011) Estimation of rearrangement phylogeny for cancer genomes, \emph{Genome Research}, {\bf 22(2)}:346-61.
\bibitem{Bentley} Bentley JL, 1975,  Multidimensional binary search trees used for associative searching, \emph{Comm ACM}, {\bf 18}, 509-517.
\bibitem{Turtle} Turtle H and Croft WB, 1991, Evaluation of an Inference Network-Based
Retrieval Model, \emph{ACM Transactions on Information Systems}, {\bf 9(3)}, 187-222.
\bibitem{Griffiths} Griffiths RC and Marjoram P, 1996, Ancestral inference from samples of DNA sequences with recombination, \emph{J. Comp. Biol.}, {\bf 3(4)}, 479-502.
\bibitem{Kirkpatrick} Kirkpatrick B, Reshefy Y, Finucanez H, Jiangx H, Zhu B and Karp RM, 2010, Comparing Pedigree Graphs, \emph{arXiv:1009.0909v2}.
\bibitem{Automata} Allouche J. and Shallit J. 2003, \emph{Automatic Sequences, Theory, Applications, Generalizations}, CUP.
\bibitem{Posets} Neggers J and Kim H S, 1998, \emph{Basic Posets}, World Scientfic.
\bibitem{Karzanov} Karzanov A and Khachiyan L, 1991, On the Conductance of Order Markov Chains, \emph{Order}, {\bf 8}, 7-15.
\bibitem{Brightwell} Brightwell G and Winkler P, 1991, Counting linear extensions, \emph{Order}, {\bf 8(3)}, 225-242.
\bibitem{SempleSteel} Semple C and Steel M, 2009, \emph{Phylogenetics}, OUP.
\bibitem{Zeng} Zeng J, 1996, Multinomial convolution polynomials, \emph{Discrete Math.}, {\bf 160} (1–3), 219–228.
\bibitem{Kinsella} Kinsella M and Bafna V, 2012, Modelling the Breakage-Fusion-Bridge Machanism: Combinatorics and cancer Genomics, \emph{RECOMB 2012, LNBI 7262}, 148-162.
\bibitem{Greenman1} CD Greenman, SL Cooke, J Marshall, MR Stratton, PJ Campbell, 2013, Modelling Breakage-Fusion-Bridge Cycles as a Stochastic Paper Folding Process, arXiv:1211.2356.
\end{thebibliography}
\end{document}